\documentclass[10pt,twoside,reqno,centertags,english]{amsart}
\usepackage[oldstylenums]{kpfonts}
\usepackage{geometry}
\usepackage{graphicx,physics}
\usepackage{lettrine}
\usepackage{amsmath,amssymb,verbatim}
\usepackage{enumerate}
\usepackage{hyperref}
\usepackage[active]{srcltx}
\newcommand{\mres}{\mathbin{\vrule height 1.6ex depth 0pt width
		0.13ex\vrule height 0.13ex depth 0pt width 1.3ex}}

\usepackage[T1]{fontenc}
\usepackage{color}

\newtheorem{theo}{Theorem}
\newtheorem{lem}{Lemma}[section]
\newtheorem{defi}[lem]{Definition}
\newtheorem{cor}[lem]{Corollary}
\newtheorem{prop}[lem]{Proposition}

\theoremstyle{remark}
\newtheorem{remark}[lem]{Remark}

\newcommand{\eps}{\varepsilon}
\newcommand{\R}{\mathbb{R}}
\newcommand{\C}{\mathbb{C}}

\newcommand{\N}{\mathbb{N}}

\newcommand{\Mm}{\mathbb{H}^+}
\newcommand{\Pmre}{\mathcal S^+}
\newcommand{\Ppmre}{\mathcal S^{++}}
\newcommand{\Ma}{\mathbb{P}}
\newcommand{\Sm}{\mathcal{H}}
\newcommand{\Pm}{\mathcal{H}^{++}}
\newcommand{\Po}{{\mathcal{H}^+}}
\newcommand{\Pon}{{\mathcal{P}}}

\newcommand{\la}{\Lambda}
\newcommand{\p}{\partial}
\numberwithin{equation}{section}
\DeclareMathOperator{\dive}{div}
\newcommand{\rhp}{G}
\newcommand{\rd}{\mathrm{d}}
\newcommand{\der}[2]{\frac{\rd#1}{\rd#2}}

\newcommand{\Mmm}{\mathbb{P}^{++}_\la}
\newcommand{\Mml}{\mathbb{P}_\la}
\renewcommand{\S}{\mathsf{S}}
\newcommand{\dFR}{d_{FR}}
\newcommand{\dih}{d_{H}}
\newcommand{\dihs}{d_{FR	}}

\newcommand{\Smv}{\mathbb{H}}
\newcommand{\narrowcv}{\overset{\ast}{\rightharpoonup}}
\renewcommand{\u}{U}

\newcommand{\qtext}[1]{\quad\mbox{#1}\quad}
\allowdisplaybreaks
\title[Schr\"odinger problem on the Fisher-Rao space]{The Schr\"odinger problem on the non-commutative Fisher-Rao space}

\author[L.~Monsaingeon]{L\'eonard Monsaingeon}
\address[L.~Monsaingeon]{GFM Universidade de Lisboa, Campo Grande, Edif\'icio C6, 1749-016 Lisboa, Portugal
\& IECL Universit\'e de Lorraine, F-54506 Vandoeuvre-l\`es-Nancy Cedex, FRANCE
}
\email{leonard.monsaingeon@univ-lorraine.fr}

\author[D.~Vorotnikov]{Dmitry Vorotnikov}
\address[D.~Vorotnikov]{University of Coimbra, CMUC, Department of
Mathematics, 3001-501 Coimbra, Portugal}{}
\email{mitvorot@mat.uc.pt}


\begin{document}
\maketitle
\begin{abstract}
We present a self-contained and comprehensive study of the Fisher-Rao space of matrix-valued non-commutative probability measures, and of the related Hellinger space. Our non-com\-mutative Fisher-Rao space is a natural generalization of the classical commutative Fisher-Rao space of probability measures and of the Bures-Wasserstein space of Hermitian positive-definite matrices. We introduce and justify a canonical entropy on the non-commutative Fisher-Rao space, which differs from the von Neumann entropy. We consequently derive the analogues of the heat flow, of the Fisher information, and of the dynamical Schr\"odinger problem. We show the $\Gamma$-convergence of the $\eps$-Schr\"odinger problem towards the geodesic problem for the Fisher-Rao space, and, as a byproduct, the strict geodesic convexity of the entropy.
\end{abstract}
\tableofcontents
\newpage
\section{Introduction}

Recently, deep connections have been established between \emph{optimal transport} and \emph{quantum information theory}, two very active and proficient fields.
The prototypical quantities involved are the quadratic \emph{Wasserstein distance} between two probability measures $\rho_0,\rho_1\in\mathcal P(\mathcal X)$ over a polish space $(\mathcal X,d)$
$$
W_2^2(\rho_0,\rho_1)=\min\limits_\pi\Bigg\{ \int_{\mathcal X\times\mathcal X}d^2(x,y)\,\rd \pi(x,y):
\qquad
\pi\in\mathcal P(\mathcal X\times\mathcal X)\mbox{ has marginals }\rho_0,\rho_1\Bigg\}
$$
and the \emph{Bures distance} between Hermitian positive semi-definite (PSD) matrices $A_0,A_1$ (or more generally two quantum operators)
$$
d_B^2(A_0,A_1)=\min\limits_R\Big\{ \left|\sqrt{ A_0} - R\sqrt{ A_1}\right|^2:\qquad R\mbox{ is unitary}\Bigg\}.
$$
Clearly these two problems share a similar variational structure, where one aims at finding an optimal coupling ($\pi$ or $R$) between a given pair of points.
The Wasserstein distance plays a significant role in probability theory, partial differential equations, geometry, etc \cite{villani08oldnew,villani03topics,S15}.
The Bures distance appears in quantum information theory \cite{BZ17} as a quantum equivalent of the Fisher information metric in information geometry \cite{Ay17}, and yields a noncommutative generalization of the Hellinger distance between probability distributions (actually, the original definition of the Bures distance \cite{B69} is not restricted to the case of PSD matrices, and is applicable in much more general settings, cf. Remark \ref{noncom2}).

The connection between these objects is three-fold.
Firstly, when $\rho_0,\rho_1$ are Gaussian measures with covariance matrices $A_0,A_1$, both distances agree as $W_2^2(\rho_0,\rho_1)=d^2_B(A_0,A_1)$ up to some irrelevant multiplicative factor \cite{alv18,takatsu2011wasserstein}.
Hence the Bures distance can be considered simply as Gaussian optimal transport (at least for real, nondegenerate \emph{positive}-definite matrices).
Secondly, recent attempts have been made to develop an optimal transport theory for \emph{quantum} objects, namely measures whose values are PSD matrices (or operators) \cite{CGT18,N14,CM14,CM17,CGT17,CGGT,PCS16,MM17,GMP16}.
Finally, let us mention that in the scalar case both worlds have been unified into a single theory, by now referred to as \emph{unbalanced optimal transport} \cite{KMV16A,LMS18,CP18} and aiming at providing a transport framework between nonnegative measures with unequal masses. The unbalanced optimal transport of matrix-valued measures has recently been introduced 
in \cite{CGT18A,BV18}.

The current theory however has two significant stipulations:
First of all, the rigorous Riemannian geometric perspective \cite{M17,takatsu2011wasserstein,BJL,MMP18} only works for nonsingular matrices (a covariance matrix must be positive-definite), and the analysis is limited to finite-dimensional statistical manifolds (the space of Gaussian measures, parametrized by the finite-dimensional manifold of symmetric positive-definite matrices $\mathcal S^{++}$, endowed with the Bures Riemannian metric). 
In this work we aim at extending the framework and further studying the Bures-Wasserstein geometry in infinite-dimensional counterparts of these statistical manifolds.
We consider complex Gaussians and possibly degenerate (\emph{semi}-definite) matrices, and we treat the case of matrix-valued \emph{measures}.
Some related attempts in more restrictive settings have been made in \cite{F08} (in the context of signal processing) and recently by the second author in \cite[Remark 2.9]{BV18}.
PSD-valued densities naturally arise in signal processing, geometry (Riemannian metrics) and other applications.
Our two spaces of interest here will be the Hellinger space of PSD-valued measures, and its Fisher-Rao subspace consisting in matrix-measures whose scalar trace integrates to one.
We think of the Hellinger space as a noncommutative version of nonnegative measures, while the normalized Fisher-Rao space can be thought of as the subspace of noncommutative \emph{probability} measures. This complies with the free probability theory, cf. Remark \ref{noncom}, and with the theory of $W^*$- and $C^*$-algebras (in particular, with the original contribution of Bures \cite{B69}), cf. Remark \ref{noncom2}.
It is worth pointing that both spaces are built upon \emph{complex} Hermitian matrices, but will be considered here as \emph{real} (formal) manifolds, cf. \cite{BJL}. 
\\

Our first contribution in this paper is a self-contained and comprehensive study of these two spaces, in particular we will reveal a very geometric structure by showing that the non-commutative Hellinger space is a \emph{metric cone} over the Fisher-Rao subspace.
In other words, the Fisher-Rao space can be viewed as a unit sphere in the ambient Hellinger space, and we shall sometimes accordingly speak of the \emph{spherical} Fisher-Rao space as opposed to the \emph{conic} Hellinger space.
In the commutative case (matrices of size $1$), our Hellinger and Fisher-Rao distances coincide with the classical Hellinger and Fisher-Rao distances \cite{Ay17,KLMP,LM17}. 
Some of our results are perhaps known to experts in the field for finite-dimensional Bures manifolds, but in the infinite-dimensional measure-valued setting the analysis requires significant technical work and we could not find the corresponding statements anywhere in the literature. 
Our starting point will consist in rewriting the \emph{static} Bures-Wasserstein distance as a \emph{dynamical} and more geometric problem, which is nothing but the celebrated Benamou-Brenier formulation \cite{BenamouBrenier00} of optimal transport restricted to Gaussian measures.
The resulting minimization immediately extends to the infinite-dimensional setting, thus giving rise to a geodesic problem: the minimization of an $L^2$ kinetic energy in the space of matrix-measure-valued curves, computed with respect to an infinite-dimensional version of the Bures metric.
This shows that the Hellinger space is a (formal) Riemannian manifold endowed with this extended quantum Fisher-Rao metric, and the Fisher-Rao space is simply a submanifold with induced Riemannian metric. 
This Riemannian structure allows for Riemannian computations in the spirit of Otto \cite{otto01}.
\\

A physical quantity often appearing both in quantum information theory and in optimal transport is entropy.
Entropy conveys significant information about the underlying geometric structures, in the sense that is is canonically associated with a corresponding geometric heat flow.
This heat flow is nothing but the (negative) gradient-flow of the entropy, and in optimal transport the groundbreaking paper \cite{JKO} led to whole variety of results ranging from applied PDEs \cite{S15} to Lott-Sturm-Villani synthetic curvature theory \cite{lott2009ricci}.
In the Bures-Wasserstein setting it turns out that the relevant notion of entropy is not the usual von Neumann entropy from quantum statistical mechanics, but is rather induced by the restriction of the Boltzmann-Shannon entropy (Kullback-Leibler divergence) to Gaussian optimal transport \cite{B13,M17}. In the infinite-dimensional case, the corresponding entropy is very much related to the classical Itakura-Saito divergence \cite{I68,M85,JNG12} from signal processing. 

Over the last few years, a particular regularization of the optimal transport problem, the so-called \emph{entropic optimal transport}, has received considerable attention and allowed for efficient numerical computations and theoretical advances \cite{peyre2019computational,carlier2017convergence,cuturi2013sinkhorn}.
This regularization is related to the Schr\"odinger problem \cite{CGP20,C14,CGP16,leonard2012schrodinger} and to \emph{Euclidean quantum dynamics} \cite{albeverio1989euclidean,Z86}, and can be considered as a blurred version of deterministic optimal transport.
Based on our dynamical framework and inspired from the very generic Schr\"odinger problem discussed formally in \cite{L19}, our second main contribution in this paper consists in justifying an adapted definition of a canonical entropy on the non-commutative Fisher-Rao space, the derivation of the analogous heat flow and Fisher information, and of the corresponding \emph{dynamical} Schr\"odinger problem.
Let us point out that a related \emph{static} Schr\"odinger problem has been considered very recently in \cite{ja20}, but only on the finite-dimensional statistical manifold of real non-singular matrices $\Ppmre$. We would also like to mention the contributions \cite{CPS18,VP19} that discuss an entropic regularization of the static commutative unbalanced optimal transport. A related regularization of the static non-commutative unbalanced optimal transport that employs the von Neumann entropy was suggested in \cite{PCS16}. 

Given a fixed functional on a Riemannian manifold, one can construct \cite{Kh20} two canonical evolutionary processes:
the associated gradient flow (a dissipative system) and Newton's equation (a Hamiltonian system).
The dynamical Schr\"odinger problem can be viewed as a third ``sibling'' in this geometric family, since it merely suffices to fix a functional (entropy) in order to define it, cf. \cite{L19}.
Moreover, any trajectory of the gradient flow solves an appropriate Schr\"odinger problem,  cf. \cite{MTV}.
On the other hand, a Schr\"odinger problem may be viewed as a Newton equation driven by a suitably defined analogue of the Fisher information.
In this connection, let us mention the discussion of several Newton equations on the commutative Fisher-Rao 
space in \cite{Kh20}. 

Our second main contribution will therefore consist in the study of the Schr\"odinger problem on the Fisher-Rao space.
We will prove that, as the temperature parameter $\eps\to 0$, the $\eps$-Schr\"odinger functional Gamma-converges towards the Fisher-Rao $L^2$ kinetic functional.
As a consequence the Schr\"odinger minimizers converge to geodesics.
An indirect byproduct of our explicit construction of the recovery sequences for the $\Gamma$-limit will be the $\frac 12$-geodesic convexity of the entropy in the Fisher-Rao space.
This has been studied in \cite{M17} on the finite-dimensional cone space by direct computations involving second order derivatives, but our additional mass constraint and the measure-theoretic setting both make the analysis more delicate here.
We note that our construction of recovery sequences for the $\Gamma$-convergence is reminiscent of \cite{baradat2020small}, and we will treat abstract metric spaces in our upcoming work \cite{MTV}.
\\

The paper is organized as follows:
In section~\ref{sec:notations} we fix the notations and define the basic concepts to be used throughout.
Section~\ref{s:prem} contains preliminary material on general metric cones, the finite-dimensional Bures-Wasserstein distance, and the associated entropy, Fisher information, heat-flow, and Schr\"odinger problem.
We then proceed with our comprehensive study of the general matrix-valued-measure setting and we introduce the Hellinger space in Section~\ref{sec:Hellinger}.
Section~\ref{sec:spherical} is concerned with the Fisher-Rao space, and also details the cone structure.
In Section~\ref{sec:spherical_heat_Schro} we discuss the spherical Fisher-Rao Riemannian structure, we explain how to carry out the corresponding variant of the Otto calculus and how to compute Fisher-Rao gradients of functionals of measures.
We also define our canonical entropy, we study its associated heat flow, and we show that the corresponding Schr\"odinger problem is well-posed.
Our last Section~\ref{sec:GCV_convex} focuses on the $\Gamma$-convergence of the $\eps$-Schr\"odinger problem towards the geodesic problem as well as the geodesic convexity of the entropy.
Finally, we opted for postponing some technical statements and proofs to the Appendix~\ref{sec:appendix}.

\section{Notations and conventions}
\label{sec:notations}
 We will use the following basic notation: \begin{itemize}
 \item
 The space $\Omega$ is a fixed separable, locally compact, metrizable topological space.
 \item
 $\C^{d\times d}$ is the space of $d\times d$ complex matrices, equipped with the \emph{real} Frobenius product and norm
 $$
 \Phi:\Psi=\Re \, \tr(\Phi^*\Psi)
 \qqtext{and}
 |\Phi|_2=\sqrt {\Phi:\Phi}.
 $$
\item 
$A^{Sym}:=\frac 1 2 (A+A^*)$ will denote the Hermitian part of $A\in \C^{d\times  d}$.
\item 
$I\in \C^{d\times d}$ is the identity matrix.
\item
$
\Sm
$ is the subspace of Hermitian $d\times d$ matrices.
\item
$
\Po
$ is the subspace of Hermitian positive-semidefinite (PSD) $d\times d$ matrices.
\item
$
\Pm
$ is the subspace of Hermitian positive-definite $d\times d$ matrices.

\item
$\mathcal S$ is the subspace of real symmetric $d\times d$ matrices.

\item
$\Pmre$ is the subspace of real symmetric positive-semidefinite matrices of size $d$.
\item
$\Ppmre$ is the subspace of real symmetric positive-definite $d\times d$ matrices.
\item
$\Pon$ is the subspace of Hermitian PSD matrices of size $d$ and of unit trace.
\item
For $A\in \Ppmre$ we write $\mathcal N(A)=\mathcal N(0,A)$ for the Gaussian distribution on $\mathbb R^d$ with mean $0$ and covariance $A$.
%
\item
We recall that for $U, V\in \Sm$ one has $\tr(UV)\in \R$.
If the matrices $U,V$ are positive-semidefinite, then $\tr(UV)\geq 0$.
By the Cauchy-Schwarz inequality,
$
|gU:V|\leq |gU:U|^{\frac 12} |gV:V|^{\frac 12}
$ 
for all $U,V\in\Sm$ and a PSD matrix $g\in\Po$.
We recall moreover the elementary inequality
\begin{equation}
\label{eq:gUU_leq_trgU2}
 \left|gU:U\right|\leq (\tr\, g) |U|_2^2.
\end{equation}
\item
The square root $\sqrt g$ of a PSD matrix $g\in\Po$ is $\sqrt{g}=R\,\operatorname{diag}(\sqrt{D_{ii}})R^*$, where $D=\operatorname{diag}(D_{ii})$ and $g=RDR^*$ is a spectral decomposition of $g$.
\item
The total variation norm of an $\Sm$-valued Radon measure $G$ on $\Omega$ is
$$
\|G\|_{TV}:=\sup\limits_{\Psi\in C_b(\Omega;\Sm),\|\Psi\|_\infty\leq 1} \left|\int_\Omega \rd G(x):\Psi(x)\right|,
$$
were $\|\Psi\|_\infty=\sup\limits_{x\in\Omega} |\Psi(x)|_2$ is computed relatively to the Frobenius norm on $\Sm$.
Clearly this norm is equivalent to $\sum_{i,j}\|G_{ij}\|_{TV}$, where each $\|G_{ij}\|_{TV}$ is the total variation norm (in the usual sense) of the complex measure $G_{ij}$.
Note that only Hermitian test-functions are needed since $G$ itself is Hermitian.
%
%
\item
$\Smv$ is the set of $\Sm$-valued \emph{finite} Radon measures $\rhp$ on $\Omega$, i.e. with $\|\rhp\|_{TV}<+\infty$.
\item
$\Mm\subset \Smv$ is the set of $\Po$-valued finite Radon measures $\rhp$ on $\Omega$.
\item
$\Ma\subset\Mm$ is the set of $\Po$-valued measures $\rhp$ on $\Omega$ with $\tr (\rhp(\Omega))=1$.
These will be our main objects, and can be considered as non-commutative probability measures (hence the notation), cf. Remarks~\ref{noncom}, \ref{noncom2}. 
\item
$\Mml\subset \Ma$ is the set of $\Po$-valued measures $\rhp$ on $\Omega$ with unit trace and absolutely continuous w.r.t. to a reference scalar measure $\la$, in the sense that $\gamma=\tr \rhp\ll\la$.
\item
$\Mmm\subset \Mml$ is the set of $\Pm$-valued measures $\rhp\in \Mml$.
\item
For $\rhp\in\Mm$ we denote the trace measure $\gamma:=\tr \rhp$, which is a nonnegative scalar measure.
The mass of $\rhp$ is $m=\|\gamma\|_{TV}=\gamma(\Omega)$.
Note that the (trace) Schatten 1-norm $|g|_1:=\tr(\sqrt{gg^*})=\tr g\geq C_d|g|_2$ controls the Frobenius norm on $\Po$, hence the mass controls $\|\rhp\|_{TV}\lesssim m$ on $\Mm$.
\item
Whenever $\lambda$ is a positive scalar measure and $\rhp\in \Mm$ we abuse the notations and write $\rhp\ll\lambda$ or $\rhp\ll\lambda I$ for $\gamma=\tr\rhp\ll\lambda$.
(This is \emph{not} equivalent to requiring that $|\rhp_{ij}|\ll\lambda$ for all $i,j$.)
In particular we always have $\rhp\ll \gamma=\tr \rhp$.
The Radon-Nikodym density $g(x)=\der{\rhp}{\gamma}(x)\in\Pon$ is unitary in the sense that automatically $\tr g(x)\equiv 1$.
\item 
We say that a sequence of matrix-valued measures $\rhp^k\in\Mm$ converges \emph{strongly}
$$
\rhp^k\to\rhp
\qqtext{if}
\|\rhp^k-\rhp\|_{TV}\to 0.
$$
\item 
We say that a sequence of matrix-valued measures $\rhp^k\in\Mm$ converges weakly-$*$
$$
\rhp^k\narrowcv \rhp
\qqtext{if}
\int_\Omega\varphi:\rd\rhp^k\to \int_\Omega\varphi:\rd\rhp
$$
for all $\varphi\in C_0(\Omega;\Sm)$, the closure (w.r.t. the uniform norm) of compactly supported functions $C_c(\Omega;\Sm)$.
(This is the predual convergence on $\Mm\subset(C_0(\Omega;\Sm) )^*$.)
 \item
For curves $t\in[0,1]\mapsto \rhp_t\in \Mm$ we write  $\rhp\in C([0,1];\Mm_{TV})$ for the continuity with respect to the strong TV topology.
We should anticipate at this stage that we will define two $\dih,\dFR$ (Hellinger and Fisher-Rao) distances on $\Mm,\Ma$, but the corresponding topologies will always be sequentially equivalent to $TV$.
As a consequence all curves will be continuous indistinctly w.r.t. any of these topologies and we will often omit the index for brevity.
\item
Whenever $\rhp\in \Mm$ one can define the $L^2(\rd\rhp;\Sm)$ space of $\Sm$-valued functions through $\|U\|_{L^2(\rd\rhp)}^2:=\int_\Omega\rd\rhp U: U$.
The Hermitian inner product is defined accordingly, and we refer to \cite{DR97} for further properties of noncommutative $L^p$ spaces.
\item
For a given curve $\rhp\in C([0,1];\Mm)$ the weighted $L^2(0,1;L^2(\rd {G}_t;\Sm))$ space is defined by disintegration $\rd t\otimes \rd\rhp_t$, with
$$
\|U\|_{L^2(0,1;L^2(\rd {G}_t))}^2:=\int_0^1 \int_\Omega \rd\rhp_t U_t:U_t\,\rd t.
$$
Similar definitions can be used for $L^p(0,1;L^2(\rd {G}_t;\Sm))$, $p\in[1,\infty]$.
\item
In a given metric space $(X,d)$ we say that a curve $x:t\in[0,1]\mapsto x_t\in X$ is $AC^p([0,1];X)$ (absolutely continuous) if there exists an $L^p(0,1)$ function $\eta(\tau)\geq 0$ such that
$$
d(x_s,x_t)\leq \int_s^t \eta(\tau)\,\rd \tau,
\hspace{2cm} \forall\, 0\leq s\leq t\leq 1.
$$
In this case the \emph{metric derivative}
$$
|\dot x_t|:=\lim\limits_{\delta\to 0}\frac{d(x_t,x_{t+\delta})}{\delta}
$$
exists almost everywhere, and it is the smallest function $\eta(\tau)$ satisfying the above inequality, \cite[thm. 1.1.4]{AGS06}.

\item
Note that, contrarily to the time-varying weighted Lebesgue spaces $L^p(0,1;L^q(\rd\rhp_t))$ whose elements are only defined $\rd t$-a.e., continuity and absolute continuity are \emph{global} notions defined up to the boundaries $t=0,1$.
For the local continuity, and whenever required, we shall write $C((0,1);X)$ and $AC^p_{loc}((0,1);X)$ to emphasize the distinction.
Unless otherwise specified, $C,AC^p$ always means \emph{globally} in time $t\in[0,1]$.

\item 
By \emph{geodesics} we always mean constant-speed, minimizing metric geodesics.
\item
We will use three separate distances $\dih,\dFR,TV$ on the spaces $\Mm,\Ma$, and the corresponding metric speeds $|\dot \rhp_t|_{H},|\dot \rhp_t|_{FR},|\dot \rhp_t|_{TV}$ of a single curve $(\rhp_t)_{t\in[0,1]}$ may {\it a priori} differ.
We will try to emphasize the difference by writing $AC^p([0,1];\Ma_{H})$, $AC^p([0,1];\Ma_{FR})$, $AC^p([0,1];\Ma_{TV})$ depending on the context. 
($H$ and $FR$ will however coincide.)
\item 
$C$ is a generic positive constant.
%
\end{itemize}
Finally, we fix once and for all a reference scalar, non-negative Radon measure $\la$ on $\Omega$ such that $\tr (\Lambda I)$ is a probability measure. 
(The reader might think of $\la$ as of a normalized Lebesgue measure.)


\section{Preliminaries} \label{s:prem}

\subsection{Metric cones}

We recall \cite{BH99,Bur} that, given a metric space $(X,d_X)$ of diameter $\leq \pi$, one can define another metric space $(\mathfrak C(X),d_{\mathfrak C(X)})$, called a \emph{cone} over $X$, in the following manner.
Consider the quotient $\mathfrak C(X):=X\times [0,\infty)/X\times \{0\}$, that is, all points of the fiber $X\times \{0\}$ constitute a single point of the cone called the apex.
In other words, points in $\mathfrak C(X)$ are of the form $[x,r]$ and we identify points $[x_0,0]\sim[x_1,0]$ for all $x_0,x_1\in X$.
Now set 
\begin{equation}
\label{cone}
d_{\mathfrak C(X)}^2([x_0,r_0],[x_1,r_1]):=r_0^2+r_1^2-2r_0r_1\cos (d_{X}(x_0,x_1)).
\end{equation}
Very few metric spaces are actually cones, and this property provides neat scaling and other nice geometric features \cite{LM17}.
A particularly regular situation appears when the diameter of $X$ is strictly less than $\pi$, since in this case there is a one-to-one correspondence between the geodesics in $X$ and $\mathfrak C(X)$.
Given a cone $Y=\mathfrak C(X)$, $X$ can be embedded canonically as a \emph{sphere} into $Y$ via $X\cong \{[x,1]:\, x\in X\}\subset\mathfrak C(X)$. In particular, $$d_{\mathrm{conic}}(x_0,x_1):=d_{\mathfrak C(X)}([x_0,1],[x_1,1])$$ defines a distance on the sphere $X$. 
On the sphere, the relation \eqref{cone} may be inverted: 
\begin{equation}
\label{cone2}
d_{X}(x_0,x_1)= \arccos \left (1-\frac 1 2 d^2_{\mathrm{conic}}(x_0,x_1)\right).
\end{equation}

\begin{lem}[\cite{BV18}]
\label{charsp}
If  $X$ is a length space, and $Y=\mathfrak C(X)$, then the distance $d_{X}(x_0,x_1)$ coincides with the infimum of $Y$-lengths of continuous curves $([x_t,1])_{t\in[0,1]}$ joining $[x_0,1]$ and $[x_1,1]$ and lying within $X\times \{1\}$. 
\end{lem}

\begin{lem}[\cite{LM17}]
\label{l:masscal}
Let $([x_t,r_t])_{t\in[0,1]}$ be a constant-speed geodesic in the metric cone $Y=\mathfrak C(X)$.
Then
\begin{equation}
r^2_t= t r_1^2+(1-t)r_0^2-t(1-t)d_Y^2([x_0,r_0],[x_1,r_1])\leq t r_1^2+(1-t)r_0^2.
\end{equation}
\end{lem}

\begin{lem}
\label{lequiv}
The distances $d_X$ and $d_{\mathrm{conic}}$ are Lipschitz-equivalent on $X$: namely, one has $d_{\mathrm{conic}}\leq d_X \leq  c d_{\mathrm{conic}}$. Here $c=\frac{\arccos(1-D^2/2)}{D}\leq \frac \pi 2$, where $D\leq 2$ is the diameter of $(X, d_{\mathrm{conic}})$.
\end{lem}

The proof is immediate by Lemma \ref{charsp} and convexity of $z\mapsto \arccos(1-z^2/2)$.
%
\subsection{The Bures-Wasserstein distance}

A notable finite-dimensional example of a metric cone structure is provided by the Bures-Wasserstein distance \cite{BJL,Bures,BZ17} on the space $\Po$ of positive-semidefinite Hermitian matrices. We recall that the Bures-Wasserstein distance can be computed in three different equivalent ways. 
The first one is the explicit formula
\begin{equation}
\label{e:burexp}
\begin{array}{rl}
d_B^2(A_0,A_1):= & \tr A_0 +\tr A_1 -2\tr \sqrt{\sqrt A_0 A_1 \sqrt A_0}
\\
= & \tr A_1 +\tr A_0 -2\tr \sqrt{\sqrt A_1 A_0 \sqrt A_1}
\end{array}
\end{equation}
Note that when $A_0$ and $A_1$ commute this reads explicitly $d_B^2(A_0,A_1)=|\sqrt{A_1}-\sqrt{A_0}|^2_2$.

The second equivalent formulation is more geometric.
\begin{prop} \label{d:bur}
The Bures-Wasserstein space is a geodesic space. Moreover, we have
\begin{equation}
\label{e:minibur} 
d_{B}^2(A_0,A_1)=\frac 1 4 \min_{\mathcal{A}(A_0,A_1)}\int_0^1 A_t U_t :U_t \rd t,
\end{equation}
where the admissible set $\mathcal{A}(A_0,A_1)$ consists of all pairs $(A_t,U_t)_{t\in [0,1]}$ such that
\begin{equation}
\label{eq:ODEb}
\left\{ 
\begin{array}{l} 
A\in C^1([0,1];\Po), U\in C^1([0,1];\Sm)
\\
A|_{t=0}=A_0;\quad A|_{t=1}=A_1,
\end{array}
\right.
\qqtext{and}
\frac {d A_t} {dt} =\left(A_t U_t\right)^{Sym}.
\end{equation}
The Bures-Wasserstein geodesics in $\Po$ correspond to minimizers of \eqref{e:minibur}.
\end{prop} 
\noindent
We believe that the claim is known to the experts in the field, at least for the non-singular matrices $A_0,\,A_1$.
We however failed to find an explicit reference (in particular for possibly degenerate matrices), and we provide an independent proof in the Appendix.
\\

The third possible formulation only works for \emph{nonsingular} (definite) matrices $A\in\Pm$, and provides an explicit relation with the quadratic Wasserstein distance $W_2$ on the space of probability measures on Euclidean spaces.
Let us start from the more classical real-valued case $A\in\Ppmre(d)$, and recall that we write $\mathcal N(A)=\mathcal N(0,A)$ for the corresponding multivariate normal distribution with mean $0\in\R^d$ and covariance $A$.
Then
\begin{prop}[\cite{alv18,takatsu2011wasserstein}]
\label{d:burw}
For any $A_0,A_1\in \Ppmre$, we have
\begin{equation}
\label{e:biw} 
d_{B}(A_0,A_1)=W_2\big(\mathcal N(A_0), \mathcal N(A_1)\big).
\end{equation}
\end{prop}
\noindent
Actually, one can go beyond \eqref{e:biw} and show that the \emph{real} subspace $\Ppmre$ of the Bures space may be viewed as a totally geodesic submanifold of the space of probability measures on $\R^d$ equipped with the Otto-Wasserstein Riemannian structure, cf. \cite{M17,ja20}.

Let us now treat the general (complex) case $A\in\Pm(d)$.
It is natural to substitute every complex entry of a  Hermitian matrix $A=(A_{jk})$ with the real $(2\times 2)$ block
\begin{equation}
\label{eq:def_inclusion_r}
 a_{jk}=x_{jk}+\mathrm{i}y_{jk}
 \longrightarrow
 \begin{bmatrix}
 x_{jk} & -y_{jk}
 \\
 y_{jk}& x_{jk}
 \end{bmatrix}
\end{equation}
This defines an inclusion function $\mathfrak r: \Pm(d)\to \Ppmre(2d)$.
Using the definition \eqref{e:burexp} it is easy to check that
\begin{equation}
\label{eq:submersion_Bures_real_complex}
d_{B}^2(\mathfrak r(A_0),\mathfrak r(A_1))=2 d_{B}^2(A_0,A_1),
\end{equation}
and employing Proposition \ref{d:burw} we immediately conclude that
\begin{prop}
\label{d:burwc}
For any $A_0,A_1\in \Pm$, we have
\begin{equation}
\label{e:biwc} 
d_{B}(A_0,A_1)=\frac 1 {\sqrt 2} W_2\big(\mathcal N(\mathfrak r(A_0)), \mathcal N(\mathfrak r(A_1))\big).
\end{equation}
\end{prop}
Consequently,  $\Pm$ may be viewed as a totally geodesic submanifold of the Otto-Wasserstein space of probability measures on $\R^{2d}$.

For any element $A\in \Po$, we set
$$
r=r(A)=\sqrt{\tr A}.
$$
Then we can identify $A$ with a pair $[A/r^2,r]\in \mathfrak C(\Pon)$.
(The first factor is normalized to unity in the sense that $\tr(A/r^2)=1$ and we think of $\Pon=\{A\in\Po:\,\tr A=1\}$ as the sphere of radius $r=1$ in $\Po$.)
The related \emph{spherical} Bures distance \cite{U92, BZ17} on $\Pon$ is defined by
\begin{equation}
\label{e:burexps}
d_{SB}^2(A_0,A_1):=\arccos \left (1-\frac 1 2d_B^2(A_0,A_1)\right),
\qquad A_0,A_1\in\Pon,
 \end{equation}
 consistently with \eqref{cone2}.
The fact that the Bures space is actually a metric cone might be well-known to the community but we never saw it explicitly written down. 

\begin{prop}
\label{p:coneb}
The space $(\Po,d_B)$ is a metric cone over $(\Pon,d_{SB})$, where $\Po$ is identified with $\mathfrak C(\Pon)$ via $A\simeq [A/r^2,r]$. 
\end{prop} 
\noindent
We omit the proof since a more general statement will be provided in Theorem \ref{th:cone}.

The one-to-one correspondence between Hermitian PSD matrices and real Gaussian distributions from Proposition \ref{d:burwc} allows to define an analogue of the Boltzmann entropy on the Bures-Wasserstein space, cf. \cite{M17}.
Indeed, the usual (negative) Boltzmann entropy $H(\rho)=\int_{\R^{2d}}\rho\log\rho$ of a multivariate Gaussian distribution $\rho=\mathcal N(A)$ reads explicitly \cite{ahm89}
$$
H(\mathcal N(A))=-\frac 1 2 (\log\det(2\pi e A)).
$$
This suggests defining the Bures entropy of a PSD matrix $A\in\Pm$ as the Boltzmann entropy of the associated Gaussian distribution
$$
E_B(A):=-\frac 1 2 (\log\det(2\pi e \mathfrak r (A))),
$$ cf. \cite{F12,B13}.
It will be more convenient for our purpose to use a ``modulated'' version ({\it \`a la} Bregman), making the entropy non-negative and attaining its minimum value (zero) at $I$: 
\begin{equation}
\label{eq:def_Entropy_modulated}
E(A):=\frac 1 2\left[ \tr (\mathfrak r(A))- \tr (\mathfrak r(I))-\log\det (\mathfrak r(A))\right]= \tr A- \tr I-\log\det A. 
\end{equation}
Note that for $A\in \Ppmre$ this is exactly the (doubled) Kullback--Leibler divergence 
$$
E(A)=2D_{KL}\left(\mathcal N(A) \,\big|\, \mathcal N(I)\right)=2\int_{\R^d} \frac{\rho_A}{\rho_I}\log\left(\frac{\rho_A}{\rho_I}\right)\,\rd\rho_I
$$
from $\rho_A=\mathcal N(A)$ to the standard normal distribution $\rho_I=\mathcal N(I)$.
Note also that this differs from the von Neumann entropy
\begin{equation*}
S(A)=\tr (A \log A)=\sum_{\lambda\in\sigma_A}\lambda\log\lambda.
\end{equation*} 
The entropy $S$ is not geodesically convex on the Bures spaces $(\Po,d_B)$ and $(\Pon,d_{SB})$.
(One can mimic the argument from \cite[Section 5.1]{LMS16} to show this for the cone $\Po$;
The case of the sphere $\Pon$ is left as an exercise for the reader.)
However, our entropy $E$ in \eqref{eq:def_Entropy_modulated} turns out to be geodesically convex both on $(\Po,d_B)$ and $(\Pon,d_{SB})$.
For $\Po$ this is proved in \cite{M17} (for real-valued matrices), and for $\Pon$ it will follow from our more general Theorem \ref{theo:1/2_convex}.

%
\subsection{The heat flow and the Schr\"odinger problem on the Bures-Wasserstein space}

The usual heat flow on $\R^d$
\begin{equation}
\label{e:heat1}
\p_t \rho_t=\Delta \rho_t 
\end{equation}
is known \cite{JKO, AGS06, villani03topics} to be the gradient flow of the Boltzmann entropy $H(\rho)=\int\rho\log\rho$ on the Otto-Wasserstein space of probability measures on $\R^d$.
It is easy to check by hand that the set of multivariate normal distributions of zero mean is invariant w.r.t. \eqref{e:heat1}. 
Thus, one can consider the restriction $\rho_t=\mathcal N(A_t)$ of the heat flow onto the real Bures subspace $\Ppmre$, cf. \cite{M17}.
The resulting evolution reads 
\begin{equation}
\label{e:gfb1}
\frac {dA_t}{dt}=2I,
\end{equation}
cf. \cite{MMP18} (the covariance grows linearly).
Similarly, the Wasserstein gradient flow of the relative entropy (Kullback--Leibler divergence) $\int_{\R^d} \rho \log \frac \rho{\mathcal N(I)}$ is  the Fokker-Planck equation 
$$
\p_t \rho_t=\Delta \rho_t-\dive \left(\rho_t \nabla \log \mathcal N(I)\right),
$$
cf. \cite{JKO, AGS06, villani03topics}.
Note that this is exactly the heat flow if $\R^d$ is viewed as a Riemannian manifold with Gaussian volume form, cf. \cite{villani08oldnew}.
One easily checks that if the initial datum $\rho_0=\mathcal N(A_0)$ is Gaussian then the solution remains Gaussian $\rho_t=\mathcal N(A_t)$, and the corresponding flow on the Bures space reads \cite{M17}
\begin{equation}
\label{e:gfb2}
\frac {dA_t}{dt}=2(I-A_t).
\end{equation}
Both \eqref{e:gfb1} and \eqref{e:gfb2} actually make sense on the whole Bures space $\Po$, i.e. even when $A$ are complex-valued and merely semi-definite.
Moreover, it can be rigorously justified (e.g., by mimicking \cite{M17}), at least if we restrict ourselves to $\Pm$, that the extended flows \eqref{e:gfb1} and \eqref{e:gfb2} are exactly the Bures-Wasserstein gradient flows of the entropies $\frac 1 2 E_B(A)$ and $\frac 1 2 E(A)$, respectively.

The production of the Boltzmann entropy $-\frac d {dt} \int_{\R^d} \rho_t \log \rho_t=\int_{\R^d} \frac {|\nabla \rho_t|^2}{\rho_t}$ along the heat flow \eqref{e:heat1} is the celebrated Fisher information.
It is a key ingredient in the formulation of the time-symmetric dynamical version of the Schr\"odinger problem \cite{CGP20, C14, CGP16, L19}:
\begin{equation}
\label{e:oldy}
\frac 1 2\int_0^1\left(\int_{\R^d} \left(\rho_t |u_t|^2+\epsilon^2 \frac {|\nabla \rho_t|^2}{\rho_t} \right) \rd x \right) \rd t
\qquad\rightarrow \min.
\end{equation}
Here the unknown probability density $\rho_t$ and velocity field $u_t$ are related by the transport equation $\p_t \rho_t+\dive (\rho_t u_t)=0$; 
the initial and final configurations $\rho_0,\rho_1$ are prescribed and $\eps>0$ is a temperature parameter.  
Just as the heat flow, the Schr\"odinger problem also leaves the the Gaussian manifold invariant:
\begin{prop}
\label{prop:Schrodinger_Gaussian}
 Assume that $\rho_0=\mathcal N(A_0),\rho_1=\mathcal N(A_1)$ are Gaussians.
 Then the solution to \eqref{e:oldy} remains Gaussian, $\rho_t=\mathcal N(A_t)$ for some explicitly computable $A_t\in\Ppmre$.
\end{prop}
\noindent
We could not find this statement anywhere in the literature and postpone a sketch of proof to the appendix.

In the spirit of \cite{L19}, this suggests considering an analogue of the dynamical Schr\"odinger problem \eqref{e:oldy} on the Bures space $\Po$, just as we have done above for the heat flow.
Related issues for the static Schr\"odinger problem were discussed very recently in \cite{ja20}.
The corresponding Fisher information, i.e. the production of the entropy $\frac 1 2 E_B(A)$ along the restricted heat flow \eqref{e:gfb1}, turns out to be  $\tr A^{-1}$.
The resulting Schr\"odinger problem reads
\begin{equation}
\label{e:minishro} 
\frac 1 2 \int_0^1 \left(A_t U_t :U_t+\eps^2 \tr A_t^{-1}\right) \rd t
\qquad \rightarrow \min,
\end{equation}
where $(A_t,U_t)_{t\in[0,1]}$ belongs to the admissible set $\mathcal{A}(A_0,A_1)$ from Proposition \ref{d:bur}, i.e. satisfy the continuity equation $\frac{dA_t}{dt}=(A_tU_t)^{Sym}$.
As the temperature $\eps\to 0$ the solutions ({\it Schr\"odinger bridges}) are expected to approximate the geodesics on $(\Po,d_{B})$, cf. \eqref{e:minibur}.
We will prove a related but more general statement later on, Corollary \ref{c:gc}.

\section{The Hellinger distance}
\label{sec:Hellinger}

The Hellinger distance on $\Mm$ can be defined as follows. 
\begin{defi} [Hellinger distance]
\label{d:fr1h}
Given two matrix measures $\rhp_0,\rhp_1\in \Mm$ we define
\begin{equation}
\label{e:minifrh}
d_{H}^2(\rhp_0,\rhp_1):=\inf_{\mathcal{A}(\rhp_0,\rhp_1)}\int_0^1\left(\int_{\Omega} \rd \rhp_t U_t :U_t \right) \rd t,
\end{equation}
where the admissible set $\mathcal{A}(\rhp_0,\rhp_1)$ consists of all pairs $(\rhp_t,U_t)_{t\in [0,1]}$ such that
\begin{equation}
\label{eq:ODE}
\left\{ 
\begin{array}{l} 
\rhp\in C([0,1];\Mm_{TV}),
\\
\rhp|_{t=0}=\rhp_0;\quad \rhp|_{t=1}=\rhp_1,
\\
U\in L^2(0,1;L^2(\rd \rhp_t;\Sm))
\end{array}
\right.
\qqtext{and}
\p_t\rhp_t=\left(\rhp_tU_t\right)^{Sym} \quad \mbox{in the weak sense}.
\end{equation}
\end{defi}
For purely aesthetic reasons we dropped a $1/4$ factor in \eqref{e:minifrh} compared to \eqref{e:minibur}.

We say that $\p_t\rhp_t=\left(\rhp_tU_t\right)^{Sym}$ holds in the weak sense if for all test-functions $\Psi\in C^1([0,1];C_b(\Omega;\Sm))$ there holds
\begin{equation}
\label{eq:weak_formulation_ODE}
\int _\Omega \Psi_t:\rd\rhp_t -\int _\Omega \Psi_s:\rd \rhp_s=
\int_s^t\left(\int_\Omega \rd G_\tau :\p_\tau \Psi_\tau + \rd G_\tau U_\tau : \Psi_\tau \right) \rd \tau,
\qquad \forall s,t\in[0,1]. 
\end{equation}
We will see shortly that this automatically implies that $G$ is absolutely continuous in time w.r.t. the TV norm, hence the ODE $\partial_t \rhp_t=(\rhp_t U_t)^{Sym}$ can also be understood as pointwise a.e. $t\in(0,1)$ with values in the Banach space $\Smv_{TV}$ and boundary data taken in the strong sense.
More precisely,
\begin{lem}
\label{lem:mass_estimate_TV}
 Let $(\rhp_t,U_t)_{t\in[0,1]}$ be a solution of \eqref{eq:weak_formulation_ODE} with $\rhp\in C([0,1];\Mm_{TV})$ and mass $m_t:=\|\tr\rhp_t\|_{TV}$.
 Set
 $$
 E:=\|U\|^2_{L^2(0,1;L^2(d\rhp_t))}
 \qqtext{and}
 M:=\max\limits_{t\in[0,1]} m_t.
 $$
 Then the map $t\mapsto\rhp_t$ is $AC^2$ for the TV norm, there holds
 \begin{equation}
 \label{eq:control_TV_leq_sqrt_FR}
  \|\rhp_t-\rhp_s\|_{TV}\leq \sqrt{ME}|t-s|^\frac 12,
  \qquad \forall\,s,t\in[0,1]
 \end{equation}
 and we have the mass estimate
 \begin{equation}
 \label{eq:fund_mass_estimate}
 M\leq E+ 2\min\{m_0,m_1\}.
 \end{equation}
\end{lem}
\begin{proof}
Note first that by continuity the mass is bounded from above, i.e. $M=\max m_t<+\infty$.
In order to estimate $M$ we first take $\psi\equiv I$ in \eqref{eq:weak_formulation_ODE}.
The Cauchy-Schwarz inequality yields
\begin{multline*}
 |m_t-m_s|
 =
 \left|\int_s^t \int_\Omega 
 \rd\rhp_\tau U_\tau:I \,\rd \tau\right|
 \\
 \leq
 \left( \int_s^t \int_\Omega 
 \rd\rhp_\tau U_\tau:U_\tau \,\rd \tau\right)^\frac 12
 \left(  \int_s^t \int_\Omega 
 \rd\rhp_\tau I:I\,\rd \tau\right)^\frac 12
 \\
 = \sqrt{E} \times \left(\int_s^t \|\tr\rhp_\tau\|_{TV}\right)^\frac 12
 \leq \sqrt{E} \times (M|t-s|)^\frac 12,
\end{multline*}
where we used  $\|\tr\rhp_\tau\|_{TV}=m_\tau\leq M$.
Taking now $s=0$ and picking any time $t$ where $M=\max m_\tau=m_{t}$ is attained gives $M\leq m_0+\sqrt{E}\sqrt M|t-0|^\frac 12\leq m_0+\sqrt{E}\sqrt M$.
Elementary algebra guarantees in turn $M\leq E+2m_0$.
By symmetry we get $M\leq E+2m_1$ as well and \eqref{eq:fund_mass_estimate} follows.

In order to get the absolute continuity in time, take now a time-independent test-function $\psi\in C_b(\Omega;\Sm)$ and fix any $s\leq t$.
Then \eqref{eq:weak_formulation_ODE} and \eqref{eq:gUU_leq_trgU2} imply
\begin{multline*}
\left|\int_\Omega \rd(\rhp_t-\rhp_s):\psi\right|
=
\left| \int_s^t\int_\Omega\rd \rhp_\tau U_\tau : \psi\,\rd \tau \right|
\\
\leq  \int_s^t\left(\int_\Omega\rd \rhp_\tau U_\tau : U_\tau\right)^\frac 12
\left(\int_\Omega\rd \rhp_\tau \psi : \psi\right)^\frac 12 
\rd \tau 
\\
\leq \int_s^t\left(\int_\Omega\rd \rhp_\tau U_\tau : U_\tau\right)^\frac 12
\left(m_\tau \|\psi\|^2_\infty\right)^\frac 12 
\rd \tau 
\\
\leq \sqrt{M}\|\psi\|_\infty \int_s^t\left(\int_\Omega\rd \rhp_\tau U_\tau : U_\tau\right)^\frac 12\rd \tau.
\end{multline*}
Taking the supremum over $\psi$'s such that $\|\psi\|_\infty\leq 1$ gives
$$
\|\rhp_t-\rhp_s\|_{TV}\leq \sqrt M \int_s^t\underbrace{\left(\int_\Omega\rd \rhp_\tau U_\tau : U_\tau\right)^\frac 12}_{\in L^2(0,1)}\rd \tau
$$
and entails the $AC^2$ regularity in total variation.
Applying finally the Cauchy-Schwarz inequality (in time) gives
$$
\|\rhp_t-\rhp_s\|_{TV}
\leq
\sqrt M \left( \int_s^t\int_\Omega\rd \rhp_\tau U_\tau : U_\tau\rd \tau\right)^\frac 12 |t-s|^\frac 12
\leq 
\sqrt{ME}|t-s|^\frac 12
$$
and concludes the proof.
\end{proof}

\begin{theo} 
\label{theo:d_distanceh}
 $d_{H}$ is a distance on $\Mm$.
\end{theo}
\begin{proof}
The argument is quite standard and we only sketch the details.
First of all, given any $\rhp_1$ it is easy to see that $(\rhp_t,U_t):=(t^2\rhp_1,\frac 2t I)$ gives an admissible path connecting $0$ to $\rhp_1$ with finite cost.
As a consequence any two points $\rhp_0,\rhp_1$ can be connected going through zero, $\rhp_0\leadsto 0\leadsto\rhp_1$ each in time $1/2$, and $\dih^2(\rhp_0,\rhp_1)$ is therefore always finite.

Assume now that $\dih^2(\rhp_0,\rhp_1)=0$, and let $(\rhp^n_t,U^n_t)_{t\in[0,1]}$ be a minimizing curve with energy $E_n=E[\rhp^n,U^n]\to 0$.
Then Lemma~\ref{lem:mass_estimate_TV} guarantees that
$$
\|\rhp_1-\rhp_0\|_{TV}\leq \sqrt{E^n+2\min\{m_0,m_1\}}\sqrt{E_n}\to 0
$$
hence $\rhp_0=\rhp_1$.

Finally for the triangular inequality, fix any $\rhp_0,\rhp_1\in \Mm$ and take any $\tilde\rhp\in\Mm$.
Consider two minimizing sequences $(\rhp^{n0},U^{n0})$ and $(\rhp^{n1},U^{n1})$ in the definitions of $\dih^2(\rhp_0,\tilde\rhp)$ and $\dih^2(\tilde\rhp,\rhp_1)$, respectively, both in time $t\in[0,1]$.
For any fixed $\theta\in(0,1)$ it is easy to scale the path $\rhp^{n0}$ in time $[0,\theta]$, rescale the path $\rhp^{n1}$ in time $[\theta,1]$, and concatenate them to produce an admissible path $(\check \rhp^n,\check U^n)_{t\in[0,1]}$ connecting $\check\rhp^n_0=\rhp_0 \leadsto \check\rhp^n_{\theta}=\tilde\rhp \leadsto \check\rhp^n_1=\rhp_1$.
The scaling property shows that the energy of the resulting path is
$$
\dih^2(\rhp_0,\rhp_1)
\leq E[\check\rhp^n,\check U^n] = \frac{1}{\theta}E[\rhp^{n0},U^{n0}] + \frac{1}{1-\theta}E[\rhp^{n1},U^{n1}]
$$
and therefore taking $n\to \infty$ gives
$$
\dih^2(\rhp_0,\rhp_1)
\leq
\frac{1}{\theta}\dih^2(\rhp_0,\tilde\rhp) + \frac{1}{1-\theta}\dih^2(\tilde\rhp,\rhp_1).
$$
Choosing $\theta=\frac{\dih(\rhp_0,\tilde\rhp)}{\dih(\rhp_0,\tilde\rhp)+\dih(\tilde\rhp,\rhp_1)}$ finally yields
$$
\frac{1}{\theta}\dih^2(\rhp_0,\tilde\rhp) + \frac{1}{1-\theta}\dih^2(\tilde\rhp,\rhp_1)
=
\left[\dih(\rhp_0,\tilde\rhp)+\dih(\tilde\rhp,\rhp_1)\right]^2
$$
and achieves the proof.
\end{proof}
%
Recalling that we write $d_{B}(g_0,g_1)$ for the Bures distance between $g_0,g_1\in\Po$, we have next
\begin{lem}
\label{lem:dih=int_dB}
The Hellinger distance can be computed by the formula
\begin{equation}
\label{eq:dih=int_dB}
 \dih^2(\rhp_0,\rhp_1)=4\int_\Omega d^2_{B}(\rhp_0,\rhp_1).
\end{equation}
\end{lem}
\noindent
Note that this integral is well-defined by $1$-homogeneity of $d_B^2$ (which is clear from \eqref{e:burexp}), i.e. we mean here
$$
\int_\Omega d^2_{B}(\rhp_0,\rhp_1)
:=
\int_\Omega d^2_{B}\left(\frac{\rd\rhp_0}{\rd\lambda}(x),\frac{\rd\rhp_1}{\rd\lambda}(x)\right)\rd \lambda(x)
$$
for any positive scalar measure $\lambda$ dominating simultaneously (the traces of) $\rhp_0,\rhp_1$.
(The integral is independent of the choice of $\lambda$.)
\begin{proof}
 Let first us prove that $d^2_H\leq 4\int d^2_B$.
 To this end, fix any scalar measure $\lambda$ dominating $\rhp_0,\rhp_1$, and denote the corresponding $\Po$-valued densities $g_0(x):=\frac{\rd\rhp_0}{\rd\lambda}(x),g_1(x):=\frac{\rd\rhp_1}{\rd\lambda}(x)$.
 For $\lambda$-a.e. $x\in\Omega$ fixed, consider a minimizing pair curve/potential in the dynamical definition (Proposition~\ref{d:bur}) of the Bures distance between $g_0(x),g_1(x)$:
 This is a pair $U_t=U_t(x)$ with values in $\Sm$ and $g_t=g_t(x)$ with values in $\Po$ such that
 $$
 d^2_B(g_0,g_1)=\frac 1 4\int_0^1 g_tU_t:U_t \,\rd t,
 \qquad 
 \frac{d}{dt} g_t=(g_tU_t)^{Sym}.
 $$
 Defining the matrix-valued measures
 $$
 \rd \rhp_t:= g_t(x)\cdot\rd\lambda
 $$
 it is easy to see that the weak formulation \eqref{eq:weak_formulation_ODE} of $\p_t\rhp_t=(\rhp_t U_t)^{Sym}$ is satisfied, and the curve $(G_t,U_t)_{t\in[0,1]}$ is an admissible competitor in the minimization \eqref{e:minifrh} for $d^2_H(\rhp_0,\rhp_1)$.
 By definition of $\rhp_t$ and Fubini's theorem we get
 \begin{multline*}
 d^2_H(\rhp_0,\rhp_1)
 \leq 
 \int _0^1\int_\Omega \rd \rhp_t(x) U_t(x):U_t(x) \,\rd t
 \\
 =
 \int _0^1\int_\Omega  \rd\lambda(x) g_t(x) U_t(x):U_t(x) \,\rd t
 =
 \int_\Omega\left(\int_0^1g_t(x) U_t(x):U_t(x) \,\rd t\right) \rd\lambda(x)
 \\
 =4\int_\Omega d^2_B\left(g_0(x),g_1(x)\right)\,\rd\lambda(x)
 =4\int_\Omega d^2_B(\rhp _0,\rhp _1).
 \end{multline*}
Let us now establish the reversed inequality $4\int d^2_B\leq d^2_H$.
To this end, pick a minimizing sequence $(\rhp_t^n,U^n_t)$ in the definition \eqref{e:minifrh} of $\dih^2(\rhp_0,\rhp_1)$, and fix any positive scalar measure $\lambda$ dominating simultaneously $\rhp_0,\rhp_1$.
We claim that we can assume $\rhp_t^n\ll\lambda$ as well for any intermediate time $t\in(0,1)$.
For if not, the linearity of the Lebesgue decomposition $\rhp_t^n=(\rhp_t^n)^\lambda+(\rhp^n_t)^\perp$ with respect to $\lambda$ (for any fixed time) easily shows that $\p_t\rhp_t^n=(\rhp_t^n U^n_t)^{Sym}$ is equivalent to the two separate ODEs $\p_t(\rhp_t^n)^\lambda=((\rhp_t^n)^\lambda U^n_t)^{Sym}$ and $\p_t(\rhp_t^n)^\perp=((\rhp_t^n)^\perp U^n_t)^{Sym}$.
Since $\rhp_0^\perp=\rhp_1^\perp=0$ clearly $((\rhp^n_n)^\lambda,U^n_t)_{t\in[0,1]}$ is an admissible path connecting $\rhp_0=\rhp_0^\lambda,\rhp_1=\rhp_1^\lambda$.
Of course we have $(\rhp^n_t)^\lambda\ll\lambda$ for all $t$, and
$$
\int _0^1\int_\Omega \rd (\rhp^n_n)^\lambda U^n_t:U^n_t\,\rd t
\leq \int _0^1\int_\Omega \rd \rhp^n_t U^n_t:U^n_t\,\rd t
$$
gives a lesser cost so $((\rhp^n_t)^\lambda,U^n_t)$ is a better competitor.

Thus assuming that $\rhp^n_t\ll\lambda$ for all times, and writing as before $g^n_t(x):=\frac{\rd G^n_t}{\rd\lambda}(x)$ for the corresponding densities, the measure-valued ODE simply means now
$$
\frac{d}{dt}g^n_t(x)=(g_t^n(x) U_t^n(x))^{Sym}
\qtext{with endpoints}g_0(x),g_1(x)
$$
for $\lambda$-a.e. $x\in\Omega$.
 In particular $(g_t^n,U_t^n)$ is an admissible curve connecting $g_0,g_1$ (in each fiber $x\in\Omega$), it is therefore an admissible competitor in the characterization \eqref{e:minibur} of $d^2_B(g_0(x),g_1(x))$ for $\lambda$-a.e. $x$, thus
 \begin{multline*}
 4\int_\Omega d^2_B(\rhp_0,\rhp_1)
 =
  4\int_\Omega d^2_B(g_0(x),g_1(x))\,\rd \lambda(x)
  \\
  \leq
  \int_\Omega \left(\int_0^1 g_t^n(x)U_t^n(x):U_t^n(x)\,\rd t\right)\rd \lambda(x)
  =
  \int_0^1 \int_\Omega g_t^n(x)U_t^n(x):U_t^n(x)\,\rd \lambda(x)\,\rd t
  \\
  =\int_0^1 \int_\Omega \rd \rhp_t^n U_t^n:U_t^n\,\rd t
  \xrightarrow[n\to\infty]{}d^2_H(\rhp_0,\rhp_1)
 \end{multline*}
 as desired.
\end{proof}
An immediate and important consequence of this is:
\begin{theo}[Existence of Hellinger geodesics]
\label{theo:exist_geodesicsh}
$(\Mm,\dih)$ is a geodesic space, i.e. for all $\rhp_0,\rhp_1\in \Mm$ the infimum in \eqref{e:minifrh} is always a minimum.
Moreover:
\begin{enumerate}[(i)]
 \item 
 A particular minimizer is given by $\rd\rhp_t=g_t(x)\cdot\rd\lambda$ and $U_t(x)$, where $(g_t,U_t)$ is a Bures geodesic from $g_0(x):=\frac{\rd \rhp_0}{\rd\lambda}(x)$ to $g_1(x):=\frac{\rd \rhp_1}{\rd\lambda}(x)$ in $\lambda$-a.e. fiber $x\in \Omega$ and $\lambda$ is any positive scalar measure dominating $\rhp_0,\rhp_1$.
\item
Any minimizer is a $d_{H}$-Lipschitz curve $t\mapsto\rhp_t$ such that $d_{H}(\rhp_t,\rhp_s)=|t-s|d_{H}(\rhp_0,\rhp_1)$ with potential $U\in  L^\infty(0,1;L^2(\rd \rhp_t;\Sm))$ such that $\|U_t\|^2_{L^2(\rd \rhp_t;\Sm)}=cst=d^2_{H}(\rhp_0,\rhp_1)$ for a.e. $t\in [0,1]$.
\end{enumerate}

\end{theo}
\begin{proof}
Fix $\rhp_0,\rhp_1$ once and for all.
\begin{enumerate}[(i)]
 \item 
As in the previous proof it is easy to see that the particular choice $(\rhp_t,U_t)$ with $\rd\rhp_t=g_t(x)\cdot\rd\lambda$ (and $U_t$ optimal from $g_0$ to $g_1$ for a.e. $x$) connects $\rhp_0,\rhp_1$ and is therefore an admissible competitor in \eqref{e:minifrh}.
Moreover \eqref{eq:dih=int_dB} gives
\begin{multline*}
 \dih^2(\rhp_0,\rhp_1)
 =
4 \int_\Omega d_B^2(g_0(x),g_1(x))\,\rd\lambda(x)
 =
 \int_\Omega \left(\int_0^1 g_t(x)U_t(x):U_t(x)\,\rd t\right)\,\rd\lambda(x)
 \\
 =
 \int_0^1\int_\Omega  \rd\lambda(x)g_t(x)U_t(x):U_t(x)\,\rd t
 =
\int_0^1\int_\Omega  \rd\rhp_t U_t:U_t \,\rd t
\end{multline*}
thus $(\rhp_t,U_t)$ is a minimizer.
\item
Pick any minimizer $(\rhp_t,U_t)$.
The constant-speed property easily follows from the fact that $t\mapsto \|U_t\|_{L^2(\rd\rhp_t)}$ is constant in time:
For if not, an easy arc length reparametrization (Lemma~\ref{lem:constant_speed_reparametrization} in the appendix) would give an admissible curve with strictly lesser energy.
\end{enumerate}
\end{proof}

We can now establish a comparison between TV and $\dih$, which will be technically convenient and used repeatedly in the sequel.
\begin{theo} 
\label{theo:comparison_2sides_dFR_TV}
For $\rhp_0,\rhp_1\in \Mm$ with masses $m_0,m_1$ there holds
\begin{equation}
\label{eq:comparison_2sides_dFR_TV}
\frac 1{4\sqrt{d}}\dih^2(\rhp_0,\rhp_1)
\leq
\|\rhp_1-\rhp_0\|_{TV}
\leq
\sqrt{m_0+m_1} \dih(\rhp_0,\rhp_1),
\end{equation}
and $\dih$ is topologically (sequentially) equivalent to the total variation distance on $\Mm$.
\end{theo}
Note that this estimate has the correct scaling with respect to mass, i.e. it is $1$-homogeneous (since the Hellinger distance, roughly speaking, scales as the square-root of the mass).
\begin{proof}
Let us start with the lower bound.
From \cite[theorem 1]{BJL} we have that
$$
d^2_B(g_0,g_1)\leq \left| \sqrt{g_1} - \sqrt{g_0} \right|_2^2
$$
for any $g_0,g_1\in\Po$, where $\sqrt{g}$ denotes the canonical square root of a PSD matrix $g$.
Moreover, the Powers-St{\o}rmer inequality \cite{powers1970free} reads exactly
$$
\left|  \sqrt{g_1}- \sqrt{g_0} \right|_2^2
\leq 
\left| g_1 -g_0 \right|_1,
$$
where $|g|_1=\tr( \sqrt{gg^*})=\tr g$ is the Schatten (trace) 1-norm of $g$.
Since we have $|g|_1\leq \sqrt d |g|_2$ we get immediately from Lemma~\ref{lem:dih=int_dB}
\begin{multline*}
 d^2_H(\rhp_0,\rhp_1)
=4\int_\Omega d^2_B(\rhp_0,\rhp_1)
=4\int_\Omega d^2_B(g_0(x),g_1(x))\rd\lambda(x)
\\
\leq
4\int_\Omega \left|  \sqrt{g_1}(x) - \sqrt{g_0}(x) \right|_2^2\rd\lambda(x)
\leq
4\int_\Omega \left|  g_1(x) - g_0(x) \right|_1\rd\lambda(x)
\\
\leq
4\sqrt d\int_\Omega \left| g_1(x) -g_0(x) \right|_2\rd\lambda(x)
=
4\sqrt{d} \|\rhp_1-\rhp_0\|_{TV}
\end{multline*}
as desired.
\\
Let us now turn to the upper bound.
Fix any positive scalar measure such that $\rhp_0,\rhp_1\ll\lambda$.
From Theorem~\ref{theo:exist_geodesicsh} we know that there exists a minimizing curve $(\rhp_t,U_t)_{t\in[0,1]}$ in \eqref{e:minifrh} with $\rd\rhp_t=g_t(x)\cdot\rd\lambda$ and $(g_t,U_t)$ a Bures geodesic from $g_0=\frac{\rd\rhp_0}{\rd\lambda}$ to $g_1=\frac{\rd\rhp_1}{\rd\lambda}$ in a.e. every fiber.
It follows from Lemma \ref{l:masscal} that $\tr g_t\leq \max\{\tr g_0, \tr g_1\}$ for all $t\in [0,1]$ along this geodesic and for $\lambda$-a.e. $x$.
Integrating and using the very rough bound $\max\{a,b\}\leq a+b$ gives here
$$
m_t
=\int_\Omega \tr g_t(x)\,\rd\lambda(x)
\leq 
\int_\Omega \left\{\tr g_0(x)+\tr g_1(x)\right\}\,\rd\lambda(x)
=
m_0+m_1.
$$
Our estimate \eqref{eq:control_TV_leq_sqrt_FR}, with here the mass control $m_t\leq M\leq m_0+m_1$ and the energy of the geodesic $E=\|U\|^2_{L^2(0,1;L^2(\rd\rhp_t))}=\dih^2(\rhp_0,\rhp_1)$, gives
$$
\|\rhp_1-\rhp_0\|_{TV}\leq \sqrt{ME}\leq \sqrt{m_0+m_1} \dih(\rhp_0,\rhp_1)
$$
as desired.
\\
For the topological equivalence, note that the lower bound in \eqref{eq:comparison_2sides_dFR_TV} immediately shows that TV is stronger than $\dih$.
Conversely, assume that $\dih(\rhp^n,\rhp)\to 0$.
We first claim that the masses remain bounded, i.e. $m^n=\tr \rhp^n(\Omega)\leq M$ for some $M$.
Since $\dih(\rhp^n,\rhp)\leq 1$ for $n$ large enough and $|m^n-m|\leq \|\rhp^n-\rhp\|_{TV}$ (test $\psi=I$), \eqref{eq:comparison_2sides_dFR_TV} gives
$$
|m^n-m|\leq \|\rhp^n-\rhp\|_{TV}\leq \sqrt{m^n+m}\dih(\rhp^n,\rhp)\leq \sqrt{m^n+m},
$$
which then guarantees the boundedness of $\{m^n\}$ as claimed.
One last use of \eqref{eq:comparison_2sides_dFR_TV} finally gives $\|\rhp^n-\rhp\|_{TV}\leq \sqrt{m^n+m}\,\dih(\rhp^n,\rhp)\leq \sqrt{M+m}\,\dih(\rhp^n,\rhp)\to 0$ and the proof is complete.
 \end{proof}
\begin{prop}[Upper bound of the distance]
\label{p:impr}
For every pair $\rhp_0,\rhp_1\in \Mm$ with masses $m_0,m_1$ one has
\begin{equation}
\label{impr}
d_{H}^2(\rhp_0,\rhp_1)\leq  4(m_0+m_1).
\end{equation}
\end{prop}
\begin{proof}
Take any reference measure $\lambda$ dominating $\rhp_0,\rhp_1$, and let $g_0(x),g_1(x)$ be the corresponding Radon-Nikodym densities.
From \cite[thm. 1]{BJL} we have
$$
d^2_B(g_0(x),g_1(x))\leq |\sqrt{g_1}(x)-\sqrt{g_0}(x)|^2_2
\leq
\tr g_0(x)+\tr g_1(x),
$$
where the second inequality follows from simple linear algebra.
Integrating w.r.t $\lambda$ over $\Omega$ and appealing to Lemma~\ref{lem:dih=int_dB} gives
$$
d^2_H(\rhp_0,\rhp_1)=4\int_\Omega d^2_B(g_0(x),g_1(x))\,\rd\lambda(x) \leq 4\int_\Omega (\tr g_0(x) +\tr g_1(x))\,\rd\lambda(x)=4(m_0+m_1)
$$
as desired.
\end{proof}
\begin{remark}[Optimality]
\label{r:ipmr}
The constant in \eqref{impr} is optimal since for every $\rhp\in \Mm$ with mass $m$ we have
\begin{equation}
\label{optimpr}
d_{H}^2(\rhp,0)=  4m.
\end{equation}
The proof is immediate from Lemma~\ref{lem:dih=int_dB} and \eqref{e:burexp} with $A_1=0$.
\end{remark}
\begin{lem}[Characterization of $AC^2$ curves]
\label{l:ac2curves}
A curve $G:[0,1]\to \Mm$ is $AC^2$ w.r.t. $\dih$  if and only if there exists $U\in L^2(0,1;L^2(\rd {G}_t;\Sm))$ such that the weak formulation \eqref{eq:weak_formulation_ODE} of $\partial_t\rhp_t=(\rhp_t U_t)^{Sym}$ holds, in which case $U_t$ is uniquely defined for a.a. $t\in(0,1)$ as an element of $L^2(\rd\rhp_t)$ and the $\dih$ metric speed is
$$
|\dot{G}_t|_H^2=\int_\Omega \rd G_t U_t:U_t
\qqtext{for a.a.}t\in(0,1).
$$ 
Consequently, $G:[0,1]\to \Mm$ is Lipschitz if and only if the corresponding $U$ belongs to $L^\infty(0,1;L^2(\rd {G}_t;\Sm))$, and the corresponding Lipschitz constant coincides with $\|U\|_{L^\infty(0,1;L^2(\rd {G}_t;\Sm))}$.
\end{lem}
\begin{proof}
Notice that our claim about Lipschitz curves immediately follows from the stronger $AC^2$ part of our statement so we only need to establish the latter.

Let us start with the easiest implication, and assume that $(\rhp,U)$ satisfy $\partial_t\rhp_t=(\rhp_tU_t)^{Sym}$ with $U\in L^2(0,1;L^2(\rd\rhp_t))$.
Fix any $[s,t]\subset [0,1]$.
Changing variables $\tau=s+\theta(t-s)$ and rescaling time gives an admissible curve $(\widetilde\rhp_\theta,\widetilde U_\theta)_{\theta\in[0,1]}=(\rhp_\tau,(t-s)U_\tau)_{\tau\in[s,t]}$ connecting $\widetilde\rhp_0=\rhp_s$ to $\widetilde\rhp_1=\rhp_t$.
By Lemma~\ref{lem:constant_speed_reparametrization} we can reparametrize $(\widetilde\rhp_\theta,\widetilde U_\theta)_{\theta\in[0,1]}\leadsto (\check \rhp_\theta,\check U_\theta)_{\theta\in[0,1]}$ with constant speed $\|\check U_\theta\|_{L^2(\rd\check \rhp_\theta)}\equiv cst$.
Taking into account the $(t-s)$ scaling in $\tau=s+\theta(t-s)$ gives exactly
$$
\dih^2(\rhp_s,\rhp_t)
\leq 
\int_0^1 \|\check U_\theta\|_{L^2(\rd\check\rhp_\theta)}^2\rd\theta
=
\left(\int_0^1\|\widetilde U_\theta\|_{L^2(\rd\widetilde\rhp_\theta)}\rd\theta\right)^2
=
\left(\int_s^t\|U_\tau\|_{L^2(\rd\rhp_\tau)}\rd\tau\right)^2
$$
and therefore $\rhp$ is $AC^2$ relatively to the $\dih$ distance.

Let us now turn to the converse implication, and let $(\rhp_t)_{t\in[0,1]}$ be an arbitrary $AC^2$ curve.
We first observe that $\dih^2(\rhp_t,\rhp_0)\leq \int _0^1 |\dot\rhp_\tau|^2\rd \tau<+\infty$ and \eqref{optimpr} control
$$
m_t:=\tr \rhp_t(\Omega)=\frac{1}{4}\dih^2(\rhp_t,0) \leq \frac{1}{4}\left(\dih(\rhp_t,\rhp_0)+\dih(\rhp_0,0)\right)^2=: M.
$$
Exploiting this mass bound in Theorem~\ref{theo:comparison_2sides_dFR_TV}, \eqref{eq:comparison_2sides_dFR_TV} guarantees that $\|\rhp_t-\rhp_s\|_{TV}\leq \sqrt{2M}\dih(\rhp_s,\rhp_t)$.
Hence $t\mapsto \rhp_t$ is absolutely continuous w.r.t. the $TV$ norm and $\partial_t\rhp_t$ is a well-defined and finite ($\Sm$-valued) Radon measure for a.a. $t\in(0,1)$.
For any fixed $\varphi\in C_b(\Omega;\Sm)$ we set
$$
\Phi(t):=\int_\Omega \rd \rhp_t:\varphi,
$$
and observe that $\Phi\in AC([0,1];\R)$ with
\begin{equation}
\label{eq:Phi'=int_ptG_phi}
\Phi'(t)=\int_\Omega \rd(\partial_t\rhp_t):\varphi
\qquad\mbox{for a.a. }t\in(0,1).
\end{equation}
Fix now any point of differentiability $t\in (0,1)$ of $\Phi$ (the set of such points has full measure, and is actually independent of $\varphi$).
Let $h$ be small enough, and pick from Theorem~\ref{theo:exist_geodesicsh} a geodesic $(\tilde\rhp_s,\tilde U_s)_{s\in[0,1]}$ from $\rhp_t$ to $\rhp_{t+h}$, satisfying in particular $\p_s\tilde\rhp_s=(\tilde\rhp_s\tilde U_s)^{Sym}$.
Then
\begin{multline}
\left|\frac{\Phi(t+h)-\Phi(t)}{h}\right|
=
\frac 1h \left|\int_\Omega \rd\rhp_{t+h}:\varphi  - \int_\Omega \rd\rhp_{t}:\varphi\right|
=
\frac 1h \left|\int_\Omega \rd\tilde\rhp_{1}:\varphi  - \int_\Omega \rd\tilde\rhp_{0}:\varphi\right|
\\
=
\frac 1h \left|\int_0^1 \int_\Omega\rd\tilde G_s\tilde U_s:\varphi \,\rd s\right|
\leq 
\frac 1h \left(\int_0^1 \int_\Omega\rd\tilde G_s\tilde U_s:\tilde U_s\,\rd s\right)^\frac 12 
\left(\int_0^1 \int_\Omega\rd\tilde G_s\varphi:\varphi\,\rd s\right)^\frac 12
\\
=\underbrace{\frac{\dih(\rhp_t,\rhp_{t+h})}{h}}_{:=A_h}
\underbrace{\left(\int_0^1 \int_\Omega\rd\tilde G_s\varphi:\varphi\,\rd s\right)^\frac 12}_{:=B_h}
\label{eq:diff_quotient_Phi_h}
\end{multline}
By standard properties of the metric speed \cite[thm. 1.1.2]{AGS06} the first term $A_h\to |\dot\rhp_t|_H$ as $h\to 0$.
In order to take the limit in the second $B_h$ term, it is first easy to argue as before and conclude from \eqref{optimpr} that the mass $\tilde m_s\leq \tilde M$ is bounded uniformly in $s\in (0,1)$ and $h\to 0$.
Observing also that $\dih(\tilde \rhp_0,\tilde \rhp_s)\leq \dih(\tilde \rhp_0,\tilde \rhp_1)=\dih(\rhp_t,\rhp_{t+h})\to 0$ as $h\to 0$ for all $s\in[0,1]$, \eqref{eq:comparison_2sides_dFR_TV} shows that
$$
\tilde\rhp_s \xrightarrow[h\to 0]{TV} \tilde\rhp_0=\rhp_t
\qquad\mbox{for a.a. }s\in [0,1].
$$
As a consequence for fixed $\varphi\in C_b$ the inner integral in $B_h$ converges pointwise in time
$$
\int_\Omega\rd\tilde G_s\varphi:\varphi \xrightarrow[h\to 0]{} \int_\Omega\rd\rhp_t\varphi:\varphi
\qquad\mbox{for a.a. }s\in [0,1].
$$
An easy application of Lebesgue's dominated convergence (with the previous mass bound $\left|\int \rd\tilde\rhp_s\varphi:\varphi\right|\leq \tilde m_s|\varphi|^2_\infty\leq C$ uniformly in $s$) finally gives
$$
B_h=\left(\int_0^1 \int_\Omega\rd\tilde G_s\varphi:\varphi\,\rd s\right)^\frac 12
\xrightarrow[h\to 0]{} \left(\int_0^1 \int_\Omega\rd G_t\varphi:\varphi\,\rd s\right)^\frac 12=\|\varphi\|_{L^2(\rd\rhp_t)}.
$$
Taking the limit in \eqref{eq:diff_quotient_Phi_h} with \eqref{eq:Phi'=int_ptG_phi} we get
\begin{equation}
\label{eq:estimate_pt_G_t_L2}
\left|\int_\Omega \rd(\p_t\rhp_t):\varphi\right|\leq |\dot \rhp_t|\cdot\|\varphi\|_{L^2(\rd\rhp_t)}
\qquad\mbox{for all }\varphi\in C_b(\Omega;\Sm).
\end{equation}
hence by density the linear map $\varphi\mapsto \int_\Omega\rd(\p_t\rhp_t):\varphi$ is continuous for the $L^2(\rd\rhp_t)$ norm.
The Riesz representation theorem therefore gives a unique element $U_t\in L^2(\rd\rhp_t)$ such that
$$
\int_\Omega \rd(\p_t\rhp_t):V =\int_\Omega \rd \rhp_t U_t:V
\qquad\mbox{for all } V\in L^2(\rd\rhp_t).
$$
This means of course $\p_t\rhp_t=(\rhp_t U_t)^{Sym}$ for a.a. $t\in(0,1)$, from which it is easy to check that the weak formulation \eqref{eq:weak_formulation_ODE} is satisfied.
Moreover from \eqref{eq:estimate_pt_G_t_L2} we see that
\begin{equation}
\label{eq:upper_bound_U_L2}
\|U_t\|_{L^2(\rd\rhp_t)}\leq |\dot\rhp_t|\in L^2(0,1).
\end{equation}
Fix now any Lebesgue point $t\in (0,1)$ for $\tau\mapsto \|U_\tau\|_{L^2(\rd\rhp_\tau)}$ and take $h$ small.
Since $(\rhp_\tau,U_\tau)_{\tau\in[t,t+h]}$ is an admissible curve connecting $\rhp_t$ and $\rhp_{t+h}$ we get, by definition of $\dih^2$ and after a suitable scaling in time,
$$
\dih^2(\rhp_{t},\rhp_{t+h})\leq h \int_t^{t+h} \int_\Omega \rd \rhp_\tau U_\tau:U_\tau \rd\tau.
$$
Dividing by $h^2$, taking the limit, and observing that $\frac{\dih(\rhp_t,\rhp_{t+h})}{h}\to |\dot \rhp_t|$ in the left-hand side because $\rhp\in AC^2$, we get from \eqref{eq:upper_bound_U_L2}
$$
|\dot\rhp_t|^2\leq \lim\limits_{h\to 0} \frac 1h \int_t^{t+h}\int_\Omega  \rd \rhp_\tau U_\tau:U_\tau \rd\tau=\|U_t\|^2_{L^2(\rhp_t)}\leq |\dot\rhp_t|^2
$$
and the proof is complete.
\end{proof}

\section{The Fisher-Rao distance}
\label{sec:spherical}
In this section we will mainly take interest in the space $\Ma=\{\rhp\in \Mm,\quad \tr\rhp(\Omega)=1\}$.
Recall that we wish to view $\Ma$ as a sphere of radius one in $\Mm$.
\begin{remark}
\label{noncom}
This is compatible with free probability theory (cf. \cite{Tao12}) in the following sense:
Fix $P\in \Ma$, and consider the $*$-algebra of (generalized) random matrices $\mathcal M:=L^{\infty-}(\rd P; \C^{d\times d})$ with identity $I$.
Defining the \emph{trace} (in the sense of the free probability theory) by
$
\tau(\xi):=\int_\Omega \rd P:\xi
$,
we get a non-commutative probability space $(\mathcal M, \tau)$.
This space is faithful but in general not tracial.
For $d=1$ we recover the classical (commutative) probabilistic setting. Moreover, the Fisher-Rao distance that we are ready to introduce will coincide with the classical Fisher-Rao distance \cite{Ay17,KLMP} from information geometry for $d=1$. 
%
In \cite{VB01}, a free-probabilistic analogue of the Wasserstein distance was introduced.
It would be interesting to go beyond our matricial setting and try to define a free probabilistic counterpart of the Fisher-Rao distance, but this lies completely out of the scope of this article.  
\end{remark}

\begin{remark}
\label{noncom2}
\emph{States} of $C^*$-algebras \cite{S70} are natural non-commutative counterparts of probability measures. In this connection, it is possible to define a distance between two states of a $C^*$-algebra as the infimum (along all cyclic $*$-representations of this algebra on Hilbert spaces) of the Hilbertian distances between the corresponding cyclic vectors (this distance is called the \emph{Bures distance}).  This idea goes back to Kakutani \cite{K48}, who was working in the commutative setting (which corresponds to our space $\Ma$ for $d=1$), to Bures \cite{B69}, who implemented it for the normal states of $W^*$-algebras, and to Uhlmann \cite{U76}, who considered the general case but was mainly concerned with the related concept of transition probability. The space $C_b(\Omega;\Sm)$ has  an obvious structure of a $C^*$-algebra. The elements of $\Ma$ are the states of this $C^*$-algebra. It seems possible to prove, in the spirit of Bures himself \cite[Proposition 2.7]{B69}, that our Hellinger distance (restricted to $\Ma$) coincides with the generic Bures distance as defined above (up to a multiplicative constant). It would be interesting to try to generalize the results of our paper to the abstract $C^*$-algebra framework. However, any developments in that direction, as well as the rigorous proof of the claim above, overstep the limits of this paper. 
\end{remark}

The Fisher-Rao distance on $\Ma$ can be defined as follows. 
\begin{defi}
[Fisher-Rao distance]
\label{d:fr1}
Given $\rhp_0,\rhp_1\in \Ma$ we set
\begin{equation}
\label{e:minifr}
d_{FR}^2(\rhp_0,\rhp_1):=\inf_{\mathcal{A}_1(\rhp_0,\rhp_1)}\int_0^1\left(\int_{\Omega} \rd \rhp_t U_t :U_t \right) \rd t,
\end{equation}
where the admissible set $\mathcal{A}_1(\rhp_0,\rhp_1)$ consists of all pairs $(\rhp_t,U_t)_{t\in [0,1]}$ such that
\begin{equation*}
\begin{cases}
\rhp\in C([0,1];\Ma_{TV}),
\\
\rhp|_{t=0}=\rhp_0;\quad \rhp|_{t=1}=\rhp_1,
\\
U\in L^2(0,1;L^2(\rd \rhp_t;\Sm))
\end{cases}
\qqtext{and}
\p_t\rhp_t=\left(\rhp_tU_t\right)^{Sym} \quad \mbox{in the weak sense}.
\end{equation*}
\end{defi}
Comparing with Definition~\ref{d:fr1h} it is clear that, at least formally, we view $\Ma$ as a submanifold of $\Mm$ with the induced Riemannian metric.

\begin{theo} 
 $d_{FR}$ is a distance on $\Ma$.
 \label{theo:d_distance}
\end{theo}
\begin{proof}
The identity of indiscernibles, symmetry, triangular inequality can be proved exactly as in Theorem~\ref{theo:d_distanceh} and we omit the details (the a priori bound on the masses $m_t=\tr \rhp_t(\Omega)=1$ even simplifies some parts).
The only delicate point is to check that the infimum in \eqref{e:minifr} is finite, i.e. that there is at least one admissible curve remaining in $\Ma$ while joining $\rhp_0,\rhp_1$ with finite cost.
As already mentioned, we will later on view $(\Ma,\dFR)$ as a unit sphere in the ambient cone-space $(\Mm,\dih)$, in which we already proved existence of geodesics.
The natural thing to do is therefore to project down these conic Hellinger geodesics onto the sphere, i.e. renormalize to unit masses.

To this end, take from Theorem~\ref{theo:exist_geodesicsh} a geodesic $(\rhp_t,U_t)_{t\in[0,1]}$ from $\rhp_0$ to $\rhp_1$ with $\p_t\rhp_t=(\rhp_t U_t)^{Sym}$.
By Corollary \ref{c:circarg} we control the masses from below as
$$
m_t=\tr\rhp_t(\Omega)\geq \underline m=\frac 12.
$$
(There is of course no circular argument in this anticipated use of Corollary \ref{c:circarg}, see Remark \ref{r:circarg}.)
It is then a simple exercise to check that
$$
(\tilde \rhp_t,\tilde U_t):=\left(\frac{1}{m_t}\rhp_t, U_t - \frac{\dot m_t}{m_t}I\right)
$$
satisfies $\p_t\tilde\rhp_t
=(\tilde\rhp_t\tilde U_t)^{Sym}
$.
Moreover taking the trace in $\p_t\rhp_t=(\rhp_t U_t)^{Sym}$ we know that $\dot m_t=\frac{d}{dt}\int_\Omega\rd\rhp_t :I=\int_\Omega\rd\rhp_t U_t:I$ and therefore
\begin{equation}
\label{eq:mdot_leq_mt_E}
|\dot m_t|^2=\left|\int_\Omega\rd\rhp_t U_t:I\right|^2
 \leq
 \left(\int_\Omega\rd \rhp_t U_t:U_t\right)\left(\int_\Omega\rd \rhp_t I:I\right)
 =
 m_t \int_\Omega\rd \rhp_t U_t:U_t.
\end{equation}
This allows to estimate the energy of the rescaled path as
\begin{multline*}
 \int_0^1\int_\Omega\rd\tilde\rhp_t\tilde U_t:\tilde U_t\,\rd t
 =\int_0^1\int_\Omega \frac{1}{m_t}\rd \rhp_t\left(U_t - \frac{\dot m_t}{m_t}I\right):\left(U_t - \frac{\dot m_t}{m_t}I\right)\,\rd t
 \\
 \leq \frac{2}{\underline m}\left[\int_0^1\int_\Omega \rd \rhp_t U_t:U_t\, \rd t
 + \int_0^1\int_\Omega \left|\frac{\dot m_t}{m_t}\right|^2\rd \rhp_t I:I\, \rd t\right]
 \\
 = \frac{2}{\underline m}\left[\int_0^1\int_\Omega \rd \rhp_t U_t:U_t \, \rd t
 + \int_0^1 \left|\frac{\dot m_t}{m_t}\right|^2m_t\, \rd t\right]
 \overset{\eqref{eq:mdot_leq_mt_E}}{\leq}\frac{4}{\underline m}\int_0^1\int_\Omega \rd \rhp_t U_t:U_t \, \rd t<+\infty
\end{multline*}
and the proof is complete.
\end{proof}

\begin{remark}
\label{rmk:characterization_AC2_H_FR}
Lemma \ref{l:ac2curves} holds in particular when the curve $\rhp$ takes values in $\Ma\subset\Mm$:
As a consequence the characterization of $AC^2$ curves in $(\Ma,\dFR)$ is exactly identical.
(Of course this stems from the fact that the Riemannian metric on $\Ma$ is simply induced by the overlying metric on $\Mm$.) 
\end{remark}

\subsection{Conic structure}

We are going to show that the abstract metric cone over our Fisher-Rao space $(\Ma,\dFR/2)$ coincides with the Hellinger space $(\Mm,\dih/2)$.
In other words, $\Ma$ is a unit sphere in the cone $\Mm$.
Firstly, for any element $G\in \Mm$, we set
$$
r=r(G):=\sqrt{m}=\sqrt{\tr\,\rhp (\Omega)}.
$$
Then we can identify $G$ with a pair $[G/r^2,r]\in \mathfrak C(\Ma)$, where the first factor is again normalized to unity, $\int_\Omega \tr G/r^2=1$.

\begin{theo}[Conic structure]
\label{th:cone} 
The space $(\Mm,\dih/2)$ is a metric cone over $(\Ma,\dFR/2)$, where $\Mm$ is identified with $\mathfrak C(\Ma)$ via $G\simeq [G/r^2,r]$. 
\end{theo}
Note that this covers one-point spaces $\Omega=\{x\}$ and implies Proposition \ref{p:coneb}, up to a minor dimensional scaling issue (the details of which are left to the reader).
\begin{proof}
{\it Step 1.}
We first observe that it suffices to establish the weaker claim that $(\Mm,d_{H}/2)$ is  a metric cone over \emph{some} metric space (which, due to the identification above, is nothing but $\Ma$ equipped with \emph{some} distance $d$).
Indeed, by Proposition \ref{p:impr}, for any two matrix measures $G_0,G_1\in \Ma$ one has
\begin{equation}
\label{hkbound}
\frac 12d_{H}(G_0,G_1)\leq \sqrt 2.
\end{equation}
If $(\Mm,d_{H}/2)$ is to be a cone over $(\Ma,d)$ as claimed, \eqref{hkbound} and \eqref{cone} imply that $\cos(d(G_0,G_1))\geq 0$, whence the diameter of $(\Ma,d)$ would be controlled from above by $\pi/2<\pi$.
By Theorem \ref{theo:exist_geodesicsh} $(\Mm,\dih)$ is a geodesic space, hence from \cite[Corollary 5.11]{BH99} $(\Ma,d)$ would also necessarily be a geodesic space.
Evoking Lemma~\ref{charsp} and Definition~\ref{d:fr1}, we see that $d$ should actually coincide with $d_{FR}/2$ as claimed (note that the infimum of the lengths coincides with the infimum of the $AC^2$-energies due to Lemma \ref{lem:constant_speed_reparametrization}).
\\
{\it Step 2.} 
In view of \eqref{hkbound} and \cite[Theorem 2.2]{LM17}, in order to prove the weaker claim in step~1 it suffices to establish the following scaling property that fully characterizes cones:
 \begin{equation}
 \label{scaling}
 d_{H}^2(r_0^2G_0,r_1^2G_1)=r_0r_1d_{H}^2(G_0,G_1)+4(r_1-r_0)^2,
 \end{equation}
 for all $G_0,G_1\in \Ma$ and $r_0,r_1\geq 0$. In the case $r_0r_1=0$ the claim is immediate by Remark \ref{r:ipmr} and our choice of $r^2=\tr \rhp(\Omega)=m$.
 We can thus assume that $r_0r_1>0$. 
 Consider the monotone increasing function
 $$
 a(t):=\frac {r_1 t}{(1-t)r_0+tr_1},
 $$
 and observe that
 $$
 a(0)=0,
 \qquad
  a(1)=1,
  \qquad 
  a'(t)[(1-t)r_0+tr_1]^2=cst=r_0r_1.
 $$
 We will also need its inverse function $t(a)$. 
 
 Let $(G_t,U_t)$ be any admissible path joining $G_0,G_1\in \Ma$.
 Then the rescaled path $(\tilde G_t,\tilde U_t)$ defined as
 \begin{equation}
 \label{eq:def_path_change_variables}
 \tilde G_t := [(1-t)r_0+tr_1]^2 G_{a(t)},
 \qquad
 \tilde U_t:=a'(t)U_{a(t)}+ 2\frac {(r_1-r_0)}{(1-t)r_0+tr_1}I
 \end{equation}
 connects $r_0^2G_0$ and $r_1^2G_1$ in $\Mm$.
 A straightforward computation shows that  $(\tilde G_t,\tilde U_t)$ satisfies $\p_t\tilde\rhp_t=(\tilde\rhp_t\tilde U_t)^{Sym}$.
 Testing \eqref{eq:weak_formulation_ODE} with $\Psi=\Phi_a=(r_0+(r_1-r_0)t(a))I$ and $a\in[0,1]$, we infer 
 \begin{multline}
 \label{num4.8}
 (r_0+(r_1-r_0)t(1))\int_{\Omega} \rd \rhp_{1} : I -(r_0+(r_1-r_0)t(0))\int_{\Omega} \rd \rhp_{0} : I
 \\
 - (r_1-r_0)\int_0^1 t'(a)\int_{\Omega} \rd \rhp_{a} : I \,\rd a\\= \int_0^1(r_0+(r_1-r_0)t(a))\int_{\Omega} \rd \rhp_{a} : U_{a} \,\rd a.
 \end{multline}
 
 Let us compute the energy of the path $\tilde G_t$, employing \eqref{num4.8}:
 \begin{align*}
 \int_0^1\left(\int_{\Omega} \rd \tilde \rhp_t\tilde U_t:\tilde U_t\right) \rd t
 &
 = r_0r_1\int_0^1a'(t)\left(\int_{\Omega} \rd \rhp_{a(t)} U_{a(t)}: U_{a(t)}\right) \,\rd t 
 \\
  & \hspace{1cm} + 4 (r_1-r_0)\int_0^1a'(t)(r_0+(r_1-r_0)t)\int_{\Omega} \rd \rhp_{a(t)} : U_{a(t)} \,\rd t
 \\
 & \hspace{2cm}+ 4 (r_1-r_0)^2\int_0^1\int_{\Omega} \rd \rhp_{a(t)} : I \,\rd t
 \\
 & =r_0r_1\int_0^1\left(\int_{\Omega} \rd \rhp_{a} U_{a}: U_{a}\right) \,\rd a
 \\
 & \hspace{1cm}
 + 4 (r_1-r_0)\int_0^1(r_0+(r_1-r_0)t(a))\int_{\Omega} \rd \rhp_{a} : U_{a} \,\rd a
 \\
 & \hspace{2cm}
 + 4 (r_1-r_0)^2\int_0^1 t'(a)\int_{\Omega} \rd \rhp_{a} : I \,\rd a
 \\
 & =r_0r_1 \int_0^1\left(\int_{\Omega} \rd  \rhp_t U_t: U_t\right) \rd t
 \\
 & \hspace{1cm}
 +4 (r_1-r_0)(r_0+(r_1-r_0)t(1))\int_{\Omega} \rd \rhp_{1} : I
 \\
 & \hspace{2cm}
 -4 (r_1-r_0)(r_0+(r_1-r_0)t(0))\int_{\Omega} \rd \rhp_{0} : I
 \\
  & =r_0r_1\int_0^1\left(\int_{\Omega} \rd \rhp_t U_t:U_t\right) \rd t+4(r_1-r_0)^2.
 \end{align*}
 Consequently,  $d_{H}^2(r_0^2G_0,r_1^2G_1)\leq r_0r_1d_{H}^2(G_0,G_1)+4(r_1-r_0)^2$.
 
 The opposite inequality is proved in a similar fashion: Taking any path $(\tilde\rhp,\tilde U)$ connecting $r_0^2\rhp_0$ to $r_1^2\rhp_1$ in $\Mm$ and undoing the change of variables \eqref{eq:def_path_change_variables} gives an admissible path $(\rhp,U)$ connecting $\rhp_0,\rhp_1$ in $\Ma$, whose cost can be computed explicitly as above.
 \end{proof}
 Note that we also proved along the way
 \begin{cor}
 The space $(\Ma,\dFR)$ has diameter $\le \pi$. 
 \end{cor}
Another useful consequence is

\begin{cor}[Existence of Fisher-Rao geodesics]
\label{theo:exist_geodesicsfr}
$(\Ma,d_{FR})$ is a geodesic space, i.e. for all $\rhp_0,\rhp_1\in \Ma$ the infimum in \eqref{e:minifr} is always a minimum.
Moreover any minimizer is a $d_{FR}$-Lipschitz curve such that $d_{FR}(\rhp_t,\rhp_s)=|t-s|d_{FR}(\rhp_0,\rhp_1)$ with potential $U\in  L^2(0,1;L^2(\rd \rhp_t;\Sm))$ such that $\|U_t\|^2_{L^2(\rd \rhp_t;\Sm)}=cst=d^2_{FR}(\rhp_0,\rhp_1)$ for a.e. $t\in [0,1]$.
\end{cor}

\begin{proof}
We have already proved in Theorem \ref{theo:exist_geodesicsh} that $(\Mm,\dih)$ is a geodesic space. 
Owing to the cone structure (Theorem~\ref{th:cone}), $(\Ma,d_{FR})$ is automatically a geodesic space (see the discussion in step 1 of the proof of Theorem~\ref{th:cone}).
The rest follows by Lemma \ref{l:ac2curves}.
\end{proof}

 \begin{cor} \label{c:circarg}
 Let $G_t$ be a $d_H$-geodesic in $\Mm$ joining $G_0,\, G_1\in \Ma$.  Then the corresponding masses $m_t$ are bounded from below:
\begin{equation} \label{e:circarg} m_t\geq 1-2t(1-t). \end{equation}
 \end{cor}
 The proof is immediate by Lemma \ref{l:masscal}, the bound \eqref{hkbound}, and the fact that $m_0=m_1=1$.
  \begin{remark} \label{r:circarg}
There was no circular reasoning in applying Corollary \ref{c:circarg} in the proof of Theorem \ref{theo:d_distance}, since \eqref{e:circarg} merely relies on \eqref{hkbound} and step 2 of the proof of Theorem~\ref{th:cone}, both of which have nothing to do with Theorem \ref{theo:d_distance}. 
 \end{remark}
%
\subsection{Topological properties}
The very particular conic structure automatically entails nice topological properties:
\begin{cor} 
 The distance $\dFR$ is topologically equivalent to the TV distance on $\Ma$.
 \label{cor:top}
\end{cor}

\begin{proof} It is immediate from Lemma \ref{lequiv} that the spherical and the conic distances $\dFR,\dih$ are topologically equivalent on the sphere, but by Theorem \ref{theo:comparison_2sides_dFR_TV} the latter one is topologically equivalent to the total variation distance. \end{proof}

\begin{prop}
\label{compls}
The metric space $(\Ma, d_{FR})$ is complete. 
\end{prop}
\begin{proof}
Take any $\dFR$-Cauchy sequence $\rhp^n$.
Since $\rhp^n$ has unit mass $m^n=\tr\rhp^n(\Omega)=1$ for all $n$, \eqref{eq:comparison_2sides_dFR_TV} shows that this sequence is also Cauchy for the $TV$ distance.
Since $\Ma$ is complete for the total variation we see that the $\rhp^n$ converges in $TV$, and therefore in $\dFR$ too owing to Corollary~\ref{cor:top}.
\end{proof}

\begin{cor} \label{complc} 
The metric space $(\Mm, d_{H})$ is complete. 
\end{cor}

\begin{proof}
By \cite[Proposition 5.9]{BH99} a metric cone $(\mathfrak C(X),d_{\mathfrak C(X)})$ is complete if and only if $(X,d_X)$ is complete, hence the result immediately follows from the cone structure (Theorem~\ref{th:cone}) and Proposition \ref{compls}.
\end{proof}

A last result will turn out to be useful later on for technical purposes:
\begin{lem}[Lower-semicontinuity]
\label{lem:dihs_LSC_weak*}
The distance $d_{FR}$ is sequentially lower semicontinuous with respect to the weak-$*$ topology on $\Ma$.
\end{lem}

\begin{proof}
Consider any two converging sequences of measures from $\Ma$,
$$
\rhp_0^k\xrightarrow[k\to\infty]{}\rhp_0,
\qquad
\rhp_1^k \xrightarrow[k\to\infty]{} \rhp_1\qquad \mbox{weakly-}*
$$
and assume that $\liminf \dFR(\rhp^k_0,\rhp^k_1)<+\infty$ (otherwise there is nothing to prove).
Up to extraction of a subsequence if needed we can moreover take $\lim \dFR(\rhp^k_0,\rhp^k_1)=\liminf \dFR(\rhp^k_0,\rhp^k_1)<+\infty$.
For each $k$, the endpoints $\rhp_0^k$ and $ \rhp_1^k$ can be joined by a geodesic $(\rhp^k_t,\u^k_t)_{t\in [0,1]}$, whose energies are therefore bounded as
$$
E[\rhp^k;\u^k]= d_{FR}^2(\rhp_0^k,\rhp_1^k)\leq E
$$
uniformly in $k\in \N$.

By the fundamental estimate \eqref{eq:control_TV_leq_sqrt_FR} with $m^k_t=M=1$ on $\Ma$ we get
$$
\forall\,t,s\in [0,1],\,\forall k\in \N:\qquad \|\rhp^k_s-\rhp^k_t\|_{TV}\leq C|t-s|^{1/2}.
$$
By the (classical) Banach-Alaoglu theorem, $\Ma\subset(C_0(\Omega;\Sm) )^*$ is moreover weakly-$*$ sequentially relatively compact.
The previous $\frac 12$-H\"older bound and the sequential lower semicontinuity of $d_{TV}$ with respect to the weak-$*$ convergence allow us to apply a refined version of the Arzel\`a-Ascoli theorem (Lemma \ref{L:aa} in the Appendix) to conclude that there exists a $TV$-continuous curve $(\rhp_t)_{t\in [0,1]}$ connecting $\rhp_0$ and $\rhp_1$ such that
\begin{equation}
\label{eq:pointwise_CV_w*}
\forall t\in [0,1]:\qquad \rhp^k_t\to \rhp_t\quad\mbox{ weakly-}*
\end{equation}
along some subsequence $k\to\infty$ (not relabeled here) and with $\|\rhp_s-\rhp_t\|\leq C|t-s|^\frac 12$.
Let $\mu^k=\rd t\otimes \rd\rhp^k_t$ be the matricial measure on $Q:=[0,1]\times\Omega$ defined by disintegration as
$$
\forall \,\phi\in C_c(Q;\Sm):\qquad \int_Q\phi(t,x):\rd\mu^k(t,x)
:=
\int_0^1\left(\int_{\Omega}\phi(t,.):\rd\rhp^k_t\right)\rd t.
$$
Leveraging the pointwise convergence \eqref{eq:pointwise_CV_w*} and the uniform bounds on the mass $m^k_t=\tr \rhp^k_t(\Omega)=1$, a simple application of Lebesgue's dominated convergence guarantees that
$$
\mu^k\to \mu^0\qquad \mbox{ weakly-}*\mbox{ in }{\Ma}(Q),
$$
where the finite measure $\mu^0=\rd t\otimes\rd\rhp_t\in \Ma(Q)$ is defined by duality in terms of the weak-$*$ limit $\rhp_t=\lim \rhp^k_t$ (as was $\mu^k$ in terms of $\rhp^k_t$), and, moreover,
$$
\mu^k\mres [t_0,t_1]\times \Omega \to \mu^0\mres [t_0,t_1]\times \Omega\quad \mbox{ weakly-}*
\qquad \forall\,t_0,t_1 \in[0,1].
$$
Let
$$
X\subset L^\infty(Q;\Sm)
$$
be the linear span of the functions of the form
$$
\Psi(t,x)=\Phi(t,x){\mathbf 1}_{[t_0,t_1]\times \Omega}(t,x), \quad \Phi\in  C^{1,0}_c(Q;\Sm),\quad t_0,t_1\in \mathbb Q\cap [0,1].
$$
We are going to apply a refined Banach-Alaoglu theorem (Lemma \ref{Ban} in the Appendix) on the space $X$ equipped with the norm
$\|\cdot\|:=
\|\cdot\|_{L^\infty(Q)}
$.
To this end, it is easy to see that $(X,\|\cdot\|)$ is separable.
Consider the following norms on $X$
$$\|\Phi\|_k=\left(\int_{Q} \rd\mu^k \Phi:\Phi\right)^{1/2},\qquad k=0,1,\dots,
$$
and the linear forms
$$
\varphi_k(\Phi)=\int_{Q}\rd\mu^k U^k :\Phi \,,\qquad k=1,2,\dots.
 $$ 
The weak-$*$ convergence of $\mu^k$ and the Cauchy-Schwarz inequality imply that the hypotheses of Lemma~\ref{Ban} are met with
$$
c_k:=\|\varphi_k\|_{(X,\|.\|_k)^*}\leq  \sqrt{E[\rhp^k;\u^k]}= d_{FR}(\rhp_0^k,\rhp_1^k).
$$ 
Hence, there exists a continuous functional $\varphi_0$ on the space $(X,\|\cdot\|_0)$
 such that up to a subsequence 
$$
\forall \Phi\in C^1_c(Q;\Sm):\qquad
\int_0^1\left(\int_{\Omega}\rd\rhp^k_t U^k_t :\Phi_t\right) \rd t 
\xrightarrow[k\to\infty]{} \varphi_0(\Phi)
$$
with moreover
\begin{equation}\label{e:phin}
\|\varphi_0\|_{(X,\|\cdot\|_0)^*}\leq \liminf_{k \to \infty} d_{FR}(\rhp_0^k,\rhp_1^k).
\end{equation}
Let $N_0\subset X$ be the kernel of the seminorm $\|\cdot\|_0$.
By the Riesz representation theorem, the dual $(X,\|\cdot\|_0)^*=(X/N_0,\|\cdot\|_0)^*$ can be isometrically identified with the completion $\overline {X/N_0}$ of ${X/N_0}$ with respect to  $\|\cdot\|_0$.
One can check that this completion is exactly $L^2(0,1;L^2(\rd\rhp_t;\Sm))$.

Consequently, there exists $\u\in L^2(0,1;L^2(\rd\rhp_t;\Sm))$ such that
$$
\varphi_0(\Phi)=\int_{Q}\rd\mu^0 U :\Phi =
\int_0^1\left(\int_{\Omega}\rd 
\rhp_t U_t :\Phi_t\right)\rd t
 $$
 and
 $$
 \|\u\|_{L^2(0,1;L^2(\rd\rhp_t))}=\|\varphi_0\|_{(X,\|\cdot\|_0)^*}.
 $$
Moreover, $(\rhp,\u)$ is an admissible curve joining $\rhp_0,\rhp_1$.
Indeed, the established convergences are enough to pass to the limit in the constraint \eqref{eq:weak_formulation_ODE} inside time intervals $[s,t]$ with $ s,t\in  \mathbb Q\cap[0,1]$ and $\Psi \in  C^1_c(Q; \Sm)$.
Since $G_t$ is a $TV$-continuous matrix function, an easy approximation argument shows that \eqref{eq:weak_formulation_ODE} actually holds for any $ s,t\in [0,1]$ and $\Psi \in  C^1_b(Q; \Sm)$.

Recalling \eqref{e:phin}, it remains to take into account that 
\begin{align*}
d_{FR}^2(\rhp_0,\rhp_1)\leq E[\rhp;\u]=\|\u\|^2_{L^2(0,1;L^2(\rd\rhp_t))}=\|\varphi_0\|_{(X,\|\cdot\|_0)^*}^2\leq \liminf\limits_{k\to\infty}d_{FR}^2(\rhp^k_0,\rhp^k_1).
\end{align*}
\end{proof}

\section{The spherical heat flow and Schr\"odinger problem}
\label{sec:spherical_heat_Schro}

\subsection{Otto calculus}
At least formally, it is clear from the above construction that one can view $\Ma$ as a real Riemannian manifold such that $d_{FR}$ becomes the Riemannian distance.
This is very similar to the celebrated Otto calculus  \cite{otto01,villani03topics,villani08oldnew},
in particular the tangent space at a point $\rhp\in\Ma$ is
$$
T_\rhp\Ma:=\left\{\Xi=\left(\rhp U\right)^{Sym}-G\int_{\Omega} \rd G:U\Big|
\qquad U\in L^2(\rd\rhp,\Sm)\right\}
$$
and the (squared) tangent norm reads
\begin{equation}
\label{eq:tangent_norm}
\|\Xi\|^2_{T_\rhp\Ma}:=\int_{\Omega} \rd \rhp U:U-\left(\int_{\Omega} \rd \rhp:U\right)^2
.
\end{equation}
The gradients of functionals $f:\Ma\to\R$ are 
\begin{equation}
 \nabla_{FR} f(\rhp)=\left[\rhp\frac{\delta f}{\delta\rhp}\right]^{Sym}-G\int_{\Omega} \rd G:\frac{\delta f}{\delta\rhp},
 \end{equation}
 where $\frac{\delta f}{\delta\rhp}$ stands for the usual first variation.
 We omit the details and refer to \cite[Appendix C]{BV18} for similar considerations.
 %
\subsection{The heat flow}
As motivated in Section \ref{s:prem}, cf. \eqref{eq:def_Entropy_modulated}, the relevant entropy here is not the classical von Neumann entropy
$
S(G)=\tr\int_\Omega \der{\rhp} {\la} \log \der{\rhp} {\la}\rd\la
$,
 but rather
\begin{multline} 
E(G)=\int_\Omega  \left(\tr \der{\rhp} {\la}- \tr I-\log \det \der{\rhp} {\la}\right)\rd\la
\\
=\int_\Omega -  \log \det \left(\der{\rhp} {\la}\right)\rd\la
=-\tr \int_\Omega  \log \left(\der{\rhp} {\la}\right)\rd \la.
\end{multline} 
By construction, $E$ is a non-negative functional that achieves its minimum and vanishes at $G=\Lambda I$. It is very similar  to the classical Itakura-Saito divergence \cite{I68,M85,JNG12} from signal processing. 

Here we slightly abused the notations and implicitly assumed that $G\in \Mml$ was absolutely continuous w.r.t $\Lambda$ for these formulas to make sense. 
However, $E$ can be extended from $\Ma_\la$ to the whole space $\Ma$ in the following natural way.
Indeed, let \begin{equation} \rd G=\rd G^\Lambda+\rd G^\bot=\der{\rhp^\Lambda} {\la}\rd \Lambda+\rd G^\bot \end{equation} be the (Radon-Nikodym)-Lebesgue decomposition of $G$ w.r.t. $\Lambda$.
Since the function $f(A)=-\log \det(A)$ is sublinear at infinity, its recession  function
$
f^\infty(A)
:=\lim\limits_{t\to+\infty}
\frac{f(tA)}{t}
$
is identically zero.
Employing the standard definition for a convex functional of measures, cf. \cite{Se64, DT84}, we can legitimately set 
\begin{equation} \label{e:mainentropy}
E(G):=\int_\Omega -  \log \det\left( \der{\rhp^\Lambda} {\la}\right)\rd\la
\hspace{1cm} \mbox{for all }\rhp\in \Ma,
\end{equation} 
 where the integral may be infinite.
 Note that this definition automatically makes $E$ lower semicontinuous w.r.t. the weak-$*$ convergence of measures, \cite[thm. 2.34]{ambrosio2000functions}.

As motivated in Section \ref{s:prem}, the analogue of the heat flow in our Fisher-Rao space can be defined as the gradient flow $\p_t G=-\nabla_{FR} E(G)$. 
In order to compute explicitly the latter FR gradient, note first that the Lebesgue decomposition $\rhp\mapsto \rhp^\la$ is linear.
 Accordingly, and since the first variation of $A\in\Po\mapsto \log\det A$ is $A^{-1}$, it is not difficult to check that the $L^2$-variation of $\rhp\mapsto \int_\Omega -\log \det \left(\der{G^\la}{\la}\right) \rd\la $ is the $\Sm$-valued function $x\mapsto -\left(\der{G^\la}{\la}\right)^{-1}(x)$.
 Since $\int_\Omega \rd \Lambda I: I=1$ and $\der{\rhp^\la}{\la}(x)=0$ for $\rd\rhp^\perp$-a.e. $x$, we compute explicitly
 \begin{multline*}
 \nabla_{FR} E(\rhp)
 =
 -\left[\rhp\left(\der{G^\la}{\la}\right)^{-1}\right]^{Sym}+G\int_{\Omega} \rd G:\left(\der{G^\la}{\la}\right)^{-1}
 \\
 =
- \left[\left(\la\der{\rhp^\la}{\la}+\rhp^\perp\right)\times\left(\der{G^\la}{\la}\right)^{-1}\right]^{Sym}
+ G\int_{\Omega} \left(\rd\la\der{\rhp^\la}{\la}+\rd\rhp^\perp\right):\left(\der{G^\la}{\la}\right)^{-1}
 \\
 =
 - \left[\la I +0\right]^{Sym}
 + G\left[\int_{\Omega} \rd\la \der{\rhp^\la}{\la} :\left(\der{\rhp^\la}{\la}\right)^{-1}
 +0\right]
 =
 -\Lambda I+\rhp.
 \end{multline*} 
 Thus the heat flow is
 \begin{equation}
\label{e:gf}
\p_t G= \Lambda I -G
\end{equation}
\begin{remark}
We point out that, although the entropy $E$ itself only sees the absolute continuous part $\der{\rhp^\la}{\la}$ of $G$, its gradient really does depend on the singular part as well through the full $\rhp$ term in \eqref{e:gf}.
 This might sound surprising at first, but can be explained recalling that $\rhp\in\Ma$ must satisfy the mass constraint $\tr\rhp(\Omega)=1$.
 One therefore cannot perturb the singular part without (in general) perturbing the absolutely continuous part as well in order to comply with the mass constraint, and it becomes clear that the gradient should depend on both the absolutely continuous and singular parts of $\rhp$.
 (This would be different if we worked in the conic Hellinger space, where no correction term is needed to enforce mass conservation.
 In that case the Hellinger gradient would only depend on $\rhp^\la$, not on $\rhp^\perp$.)
 \end{remark}
 
 As usual, the corresponding Fisher information is then defined as the production of the entropy $E$ along its own (negative) gradient flow \eqref{e:gf}.
 To make this more explicit, assume for simplicity that $\rhp=\rhp^\la$ is absolutely continuous.
 Then by definition we set
\begin{multline}
\label{e:fism}
F(G):=-\frac {d} {dt}E(G)=\frac d {dt}\int_\Omega  \log \det \left(\der{G}{\la}\right)\rd\la 
\\
= \int_\Omega \rd \la  \left(I-\der{G}{\la}\right):\left(\der{G}{\la}\right)^{-1}=\tr \int_\Omega \rd\la \left [\left(\der{G}{\la}\right)^{-1}-I \right].
\end{multline}
With \eqref{e:fism} at hand, observe that the function $A\mapsto A^{-1}-I$ is sublinear and therefore its recession function vanishes.
Just like we did for the entropy in \eqref{e:mainentropy}, we can lawfully extend the definition of the Fisher information to all $G\in \Ma$ by setting 
\begin{equation}
\label{eq:def_Fisher}
F(G):=\tr \int_\Omega \rd\la \left [\left(\der{G^\Lambda}{\la}\right)^{-1}-I \right],
\end{equation}
where the integral may be infinite.
This is of course consistent with what one would obtain by differentiating in time $E(\rhp)=E(\rhp^\la)$ along solutions of \eqref{e:gf}, even if $\rhp$ was not absolutely continuous.
This definition is also consistent with the usual representation
$$
F(\rhp)=\|\nabla_{FR} E(\rhp)\|^2_{T_\rhp\Ma}
$$
in terms of squared Riemannian norms.
Indeed, the above computation shows that the tangent vector $\nabla_{FR}E(\rhp)$ is represented in $T_\rhp\Ma$ by the matrix function $\left(\der{\rhp^\la}{\la}\right)^{-1}(x)$, and plugging $U=\left(\der{\rhp^\la}{\la}\right)^{-1}$ into \eqref{eq:tangent_norm} gives exactly \eqref{eq:def_Fisher}.
\\

Denote the semigroup generated by the gradient flow
\eqref{e:gf}
by
\begin{equation}
\label{eq:explicit_heat_flow}
\S_t(G):=\Lambda I +e^{-t}(G-\Lambda I).
\end{equation}
Its absolutely-continuous part is 
\begin{equation}
\label{eq:explicit_heat_flow_ac}
[\S_t(G)]^\Lambda=\Lambda I + e^{-t}(\rhp^\la-\la I)=\S_t(G^\Lambda),
\end{equation}
and the singular part is 
\begin{equation}
\label{eq:explicit_heat_flow_s}
[\S_t(G)]^\bot=e^{-t}G^\bot.
\end{equation}
The main regularization property of the heat flow that will serve our purpose is the time decay of the driving entropy, expressed here in the slightly stronger sense (in every fiber):
\begin{lem}[Exponential decay]
\label{lem:entdecay}
For any $G$ in $\Ma$ and $t\geq s\geq 0$, we have 
$$
-\log \det \left(\der{\S_t(G^\Lambda)} {\la}(x)\right)
\leq e^{-(t-s)}\left\{- \log \det \left(\der{\S_s(G^\Lambda)} {\la}(x)\right)\right\}
$$
for $\Lambda$-a.e. $x\in\Omega$.
Consequently,
$$
E(\S_t(G))\leq e^{-(t-s)} E(\S_s(G)).
$$
\end{lem}
\begin{remark}
 For abstract gradient flows in Riemannian manifolds $\dot x=-\nabla \Phi(x)$ the exponential decay $\Phi(x_t)\leq e^{-2\lambda (t-s)}\Phi(x_s)$ is often related to the $\lambda$-convexity of $\Phi$ along Riemannian geodesics.
Lemma~\ref{lem:entdecay} thus suggests that our entropy functional might be $\frac 12$-geodesically convex w.r.t. to our Fisher-Rao distance $\dFR$ on $\Ma$.
This will be proved completely rigorously later on, see Theorem~\ref{theo:1/2_convex}.
\end{remark}

\begin{proof}
Writing for simplicity $g=\der{\rhp^\la}{\la}$ and employing the concavity of the function $\log \det$, we deduce
\begin{multline*}
\log \det\left( \der{\S_t(G^\Lambda)} {\la}\right)=\log \det(I+e^{-t}(g-I))\\=\log \det((1-e^{s-t})I+e^{s-t}(I+e^{-s}(g-I)))\\ \geq e^{s-t}\log \det(I+e^{-s}(g-I))=e^{s-t} \log \det \left(\der{\S_s(G^\Lambda)} {\la}\right).
\end{multline*}
The second part of the statement immediately follows by integrating w.r.t $\la$ over $\Omega$.
\end{proof}

%
\subsection{The Schr\"odinger problem}
With the Fisher information now properly defined, the Schr\"odinger problem on $(\Ma,\dFR)$ rather classically reads, cf. \cite{L19,  CGP20, C14, CGP16, Y81, Z86}:
\begin{defi}[Schr\"odinger problem]
\label{d:yas}
Given a fixed $\epsilon>0$ and $\rhp_0,\rhp_1\in \Ma$ we define
\begin{equation}
\label{e:miniy_metric}
\mathfrak S_\epsilon(\rhp_0,\rhp_1):=
\inf\limits_{\rhp}
\left\{
\frac 1 2 \int_0^1 |\dot\rhp_t|^2_{FR}\rd t 
+ \frac {\eps^2} 2 \int_0^1 F(\rhp_t)\,\rd t,
\quad\mbox{s.t. }t\mapsto \rhp_t\in \Ma\mbox{ has endpoints }\rhp_0,\rhp_1
\right\}.
\end{equation}
\end{defi}
The metric speed in the first term is computed relatively to the distance $d_{FR}$ on $\Ma$, and $F(G_t)$ is the Fisher information just defined in \eqref{eq:def_Fisher}.
Owing to our characterization of $AC^2$ curves (Lemma~\ref{l:ac2curves}) and our explicit formula for the Fisher information we see that \eqref{e:miniy} also reads
\begin{equation}
\label{e:miniy}
\mathfrak S_\epsilon(\rhp_0,\rhp_1)
=\inf_{\mathcal{A}_1(\rhp_0,\rhp_1)}
\left\{
\frac 1 2  \int_0^1\left(\int_{\Omega} \rd \rhp_t U_t :U_t \right) \rd t
+\frac{\epsilon^2}{2} \tr\int_0^1\left(\int_{\Omega} \rd\la \left [\left(\der{G^\la_t}{\la}\right)^{-1}-I \right] \right) \rd t
\right\},
\end{equation}
where the admissible pairs $(\rhp_t,U_t)_{t\in[0,1]}\in\mathcal{A}_1(\rhp_0,\rhp_1)$ are as in Definition~\ref{d:fr1}.
 
We will prove shortly that this problem has a unique minimizer, which we call the {\it$\eps$-geodesic} or {\it Schr\"odinger bridge} between $\rhp_0,\rhp_1$.
In order to study this dynamical Schr\"odinger problem we shall need the following technical result:
\begin{prop}
\label{prop:AC2_energy_convex_LSC}
 The kinetic action functional
 $$
 \rhp\mapsto \mathcal K(G):=\frac 12\int_0^1 |\dot \rhp_t|_{FR}^2\rd t
 $$
 is convex for the linear interpolation on $C([0,1];\Ma_{FR})$, and l.s.c. for the pointwise weak-$*$ convergence (i.e. $\rhp_t^n\narrowcv \rhp_t$ for all $t\in[0,1]$).
\end{prop}
\begin{proof}
We claim that we have the dual formulation
\begin{multline}
\label{eq:dual_AC2}
\mathcal K(G) = \sup\limits_{\varphi\in C^{1,0}_b([0,1]\times\Omega;\Sm)}
\Bigg\{\int_\Omega\rd\rhp_1:\varphi_1 - \int_\Omega\rd\rhp_0:\varphi_0
\\
-\frac 12\int_0^1\int_\Omega\rd \rhp_t \varphi_t:\varphi_t\, \rd t -\int_0^1\int_\Omega \rd\rhp_t:\partial_t\varphi_t \,\rd t
\Bigg\}
\end{multline}
where we denoted $\varphi_t(x)=\varphi(t,x)$ with an obvious abuse of notations.
To see this, observe from our characterization of $AC^2$ curves (Lemma~\ref{l:ac2curves} in $\Ma$ rather than in $\Mm$, see Remark~\ref{rmk:characterization_AC2_H_FR}) that
$$
\frac 12\int_0^1|\dot \rhp_t|_{FR}^2\rd t
= \int_0^1 \int_\Omega \rd \rhp_t U_t:U_t \,\rd t
= \int_0^1 \int_\Omega (g_t U_t:U_t )\,\rd\gamma_t\rd t,
$$
where we denoted as before $g_t(x)=\frac{\rd \rhp_t}{\rd \gamma_t}(x)$ with $\gamma_t=\tr\rhp_t$.
The potential $U_t(x)$ is such that
$$
\partial_t\rhp_t=(\rhp_t U_t)^{Sym},
$$
and we recall from Lemma~\ref{lem:mass_estimate_TV} that $t\mapsto \rhp_t$ is $AC^2$ in time with values in $TV$.
In particular the ODE holds for a.e. $t$ in the Banach space $\Smv_{TV}$, and since $(\rhp_t U_t)^{Sym}\ll \gamma _t$ we see that $\partial_t \rhp_t\ll \gamma_t$ as well.
We denote
$$
\xi_t(x):=\frac{\rd (\partial_t\rhp_t)}{\rd\gamma_t}
\qqtext{and}
(g_t(x)U_t(x))^{Sym}=\frac{\rd\left((\rhp_t U_t)^{Sym}\right)}{\rd\gamma_t}
$$
the corresponding Radon-Nikodym densities, and we have thus the constraint
$$
\xi_t(x)=(g_t(x)U_t(x))^{Sym}
\qquad \mbox{for }\rd t\otimes\rd \gamma_t \mbox{-a.e. }(t,x).
$$
Now it is a simple (pointwise and finite-dimensional) exercise to check that, given $g\in \Pon$ (with unit trace) and $\xi,U\in \Sm$ with $\xi=(gU)^{Sym}=\frac 12(gU+Ug)$, we have the characterization
\begin{equation}
\label{eq:dual_AC2_pointwise}
\frac 12 gU:U=\max\limits_{\phi\in \Sm}\left\{-\frac 12 g\phi:\phi+\xi:\phi\right\}.
\end{equation}
The maximizer is of course $\phi=U$.
For any fixed test-function $\varphi=\varphi_t(x)\in C^{1,0}_b$ this gives
\begin{multline*}
 \frac 12\int_0^1|\dot \rhp_t|^2\rd t
= \int_0^1 \int_\Omega (g_t U_t:U_t \rd t)\rd\lambda_t
= \int_0^1 \int_\Omega \left(
\max\limits_{\phi\in \Sm}\left\{-\frac 12 g_t(x)\phi:\phi+\xi_t(x):\phi\right\}
\right)\rd\lambda_t(x)\rd t
\\
\geq
\int_0^1 \int_\Omega \left\{-\frac 12 g_t(x)\varphi_t(x):\varphi_t(x)+\xi_t(x):\varphi_t(x)\right\} \rd\lambda_t(x)\rd t
\\
= 
-\frac 12\int_0^1 \int_\Omega \rd\rhp_t \varphi_t:\varphi_t \,\rd t +\int_0^1\int_\Omega \rd(\partial_t\rhp_t):\varphi_t\,\rd t 
\\
=\int_\Omega\rd\rhp_1:\varphi_1 - \int_\Omega\rd\rhp_0:\varphi_0
\\
-\frac 12\int_0^1\int_\Omega\rd \rhp_t \varphi_t:\varphi_t \rd t -\int_0^1\int_\Omega \rd\rhp_t:\partial_t\varphi_t \rd t
\end{multline*}
In order to check that there is no duality gap in the above inequality, we simply approximate the optimizer in \eqref{eq:dual_AC2_pointwise}: More precisely, a density argument \cite[Thm. 2.11]{fonseca2007modern} allows to pick a sequence $\varphi^n\in C^{1,0}_b$ such that $\varphi^n\to U$ in $L^2(0,1;L^2(\rd \rhp_t;\Sm))$ and $\varphi^n\to U$ in $L^1(0,1;L^1(\rd(\partial_t\rhp_t);\Sm))$, hence
\begin{multline*}
-\frac 12\int_0^1 \int_\Omega \rd\rhp_t \varphi^n_t:\varphi^n_t \,\rd t +\int_0^1\int_\Omega \rd(\partial_t\rhp_t):\varphi^n_t\,\rd t
\\
\xrightarrow[n\to\infty]{}
-\frac 12\int_0^1 \int_\Omega \rd\rhp_t U_t:U_t \,\rd t +\int_0^1\int_\Omega \rd(\partial_t\rhp_t):U_t\,\rd t
\end{multline*}
as needed and \eqref{eq:dual_AC2} follows.

Returning now to our main statement, we see from the dual representation \eqref{eq:dual_AC2_pointwise} that the kinetic action can be written as a supremum of convex (linear) functionals, hence the convexity.
Likewise, observe that, for fixed $\varphi\in C^{1,0}_b$, the functional
$$
\rhp\mapsto 
\int_\Omega\rd\rhp_1:\varphi_1 - \int_\Omega\rd\rhp_0:\varphi_0
-\frac 12\int_0^1\int_\Omega\rd \rhp_t \varphi_t:\varphi_t \rd t -\int_0^1\int_\Omega \rd\rhp_t:\partial_t\varphi_t \rd t
$$
is continuous for the pointwise weak-$*$ convergence
(the first two boundary terms are immediate, and for the time-space integrals one can simply apply Lebesgue's dominated convergence with $m_t=\tr \rhp_t(\Omega) \equiv 1$).
As a consequence the action functional $\mathcal K$ is l.s.c. as a supremum of continuous functionals.
\end{proof}

\begin{theo}
\label{theo:exist_shr}
For fixed $\rhp_0,\rhp_1\in \Ma$ with $E(\rhp_0),E(\rhp_1)<+\infty$ the infimum in \eqref{e:miniy_metric}\eqref{e:miniy} is always a minimum.
Moreover this minimum is attained for a unique curve $\rhp^\epsilon\in AC^2([0,1];\Ma_{FR })$ and a unique potential $U^\epsilon\in  L^2(0,1;L^2(\rd \rhp_t;\Sm))$.
\end{theo}
For the usual Schr\"odinger problem it is known that $E(\rhp_0),E(\rhp_1)<+\infty$ is also a necessary condition for the well-posedness of \eqref{e:miniy} but for the sake of simplicity we omit the details, see e.g. \cite{MTV}.
\begin{proof}
Pick from Corollary~\ref{theo:exist_geodesicsfr} a Fisher-Rao geodesic $(\rhp_t)_{t\in[0,1]}$ between $\rhp_0,\rhp_1$ (for $\eps=0$).
Slightly anticipating on independent results from the next section, let $\rhp^\eps=(\rhp^\eps_t)_{t\in[0,1]}$ be the curve obtained by perturbing $G$ as in Lemma~\ref{theo:upper_bound_recovery} (here $\eps>0$ is the fixed temperature parameter in the Schr\"odinger functional).
Then \eqref{eq:upper_bound_recovery} guarantees that $\rhp^\eps$ has finite $\eps$-cost, and therefore the infimum in \eqref{e:miniy} is finite.

Choose now any minimizing sequence $\{\rhp^n\}_n$, and let us implement the direct method in the Calculus of Variations.
From \eqref{e:miniy_metric} we see that the kinetic action $\mathcal K(G^n)=\frac 12\int_0^1 |\dot \rhp^n_t|_{FR}^2\rd t$ is bounded, thus $\{\rhp^n\}_n$ is uniformly $AC^2$ and therefore equicontinuous w.r.t. the metric space $(\Ma,\dFR)$.
From a variant of the Arzel\`a-Ascoli theorem (Lemma~\ref{L:aa} in the Appendix) and the weak-$*$ lower semicontinuity from Lemma~\ref{lem:dihs_LSC_weak*} we see that there is a $\dFR$-continuous curve such that (up to extraction of a subsequence if needed)
$$
\rhp^n_t \narrowcv \rhp_t
\qqtext{for all}t\in[0,1].
$$
The lower semicontinuity in Proposition~\ref{prop:AC2_energy_convex_LSC} immediately gives
$$
\int _0^1|\dot\rhp_t|^2\rd t \leq \liminf \int _0^1|\dot\rhp^n_t|^2\rd t.
$$
As for the Fisher information $\mathcal F(\rhp)=\int_0^1 F(\rhp_t)\rd t$, note that $g\mapsto Tr\,g^{-1}$ is (linearly) strictly convex and l.s.c. on $\Po$.
As a consequence we infer from \cite[thm. 2.34]{ambrosio2000functions} that $F$ is l.s.c. for the weak-$*$ convergence on $\Mm$.
By Fatou's lemma we see that
$$
\int_0^1 F(\rhp ^n_t)\,\rd t
\leq \int_0^1 \liminf F(\rhp^n_t)\,\rd t\leq \liminf \int_0^1 F(\rhp^n_t)\,\rd t
$$
and therefore $\rhp=\lim \rhp^n$ is a minimizer.
 
Finally, the uniqueness follows from the convexity of the $AC^2$ energy in Proposition~\ref{prop:AC2_energy_convex_LSC} combined with the strict convexity of the Fisher information, and the proof is complete.
\end{proof}

\section{$\Gamma$-convergence and geodesic convexity}
\label{sec:GCV_convex}

In this section we aim at proving the $\Gamma$-convergence of the $\eps$-Schr\"odinger problem \eqref{d:yas} towards the geodesic problem \eqref{e:minifr} as $\eps\to 0$, for fixed endpoints $\rhp_0,\rhp_1\in\Ma$.
As it is often the case, the $\Gamma-\liminf$ part will not be too difficult and will rely on some suitable lower semicontinuity.
The construction of recovery sequences in the $\Gamma-\limsup$ will be technically more involved, and requires the construction of a suitable $\eps$-perturbation $(\tilde\rhp_t)_{t\in[0,1]}$ of any fixed $AC^2$ curve $(\rhp_t)_{t\in[0,1]}$.
Our construction will involve the heat flow \eqref{e:gf} as a quantitative regularizing tool, and a careful examination of the defect of optimality (at order one in $\eps$) will yield as a byproduct the $\frac 12$-geodesic convexity of the entropy.
This strategy was already exploited in \cite{baradat2020small}, and will be extended to abstract metric spaces in our subsequent work \cite{MTV}.
\\

We will first need a significant number of technical preliminaries.
For brevity, and given the Lebesgue decomposition $\rhp=\rhp^\la+\rhp^{\perp}$ of an arbitrary $\rhp\in \Ma$ with respect to the fixed reference measure $\Lambda$, we will denote below
$$
g(x)=\der{G^\Lambda}{\la}(x).
$$
We use similar lowercase notations $\tilde g,g_0$ for the Radon-Nikodym densities (w.r.t. $\la$) of other corresponding objects $\tilde G, G_0$, etc.  

In the next several lemmas, $(G_t)_{t\in[0,1]}$ is a given curve in $AC^2([0,1];\Ma_{FR})$,
and we consider a fixed Lipschitz-continuous, non-negative function $h:[0,1]\to\R^+$ with
$$
h(t)>0
\qqtext{for all}
t\in(0,1).
$$
This will allow to control a change of time scale $s=h(t)$, and $h$ will be carefully chosen later on.
Then we define the perturbed curve
\begin{equation}
\label{e:gfs}
\tilde G_t:=\S_{h(t)}G_t
\qquad \mbox{for any fixed }t\in[0,1]
\end{equation}
constructed by solving the heat-flow for a time $s=h(t)$ starting from $\rhp_t$.
\begin{lem}
\label{l:fbound}
For $\tilde G$ defined by \eqref{e:gfs} we have  $F(\tilde G)\in L^\infty_{loc}(0,1)$.
\end{lem}

\begin{proof}
For every $t\in(0,1)$ we can estimate
\begin{multline*}
F(\tilde G_t)=\tr \int_\Omega \rd\la \left [\left(I+e^{-h(t)}(g_t-I)\right)^{-1}-I \right]
\\
\leq \tr \int_\Omega \rd\la \left [\left((1-e^{-h(t)})I)\right)^{-1}-I \right]=\frac 1 {1-e^{-h(t)}}-1,
\end{multline*}
and we conclude observing that with our assumptions $h(t)>0$ is locally bounded away from zero.
\end{proof}

\begin{lem} 
\label{l:acreg}
We have  $\tilde G\in AC^2_{loc}((0,1); \Ma_{FR})$ and, moreover, $\tilde G\in C([0,1]; \Ma)$. 
\end{lem}

\begin{proof} 
Differentiating \eqref{e:gfs} in time at a.a. $t\in (0,1)$ with \eqref{eq:explicit_heat_flow} gives
$$
\p_t \tilde G_t =e^{-h(t)}\p_t G_t -h'(t)e^{-h(t)}(G_t-\Lambda I),
$$
whence 
\begin{equation}
\label{eq: pre1}
\p_t \tilde G-h'(t)\Lambda e^{-h(t)}(I-g_t)+h'(t) e^{-h(t)}G_t^\perp=e^{-h(t)}\p_t G_t.
\end{equation} 
By Lemma \ref{l:ac2curves} (characterization of $AC^2$ curves), there exists $U_t\in L^2(0,1;L^2(\rd {G}_t;\Sm))$ such that
\begin{equation}
\label{eq: expag}
\p_t G_t=(\rhp_tU_t)^{Sym}=\Lambda\left({g}_tU_t\right)^{Sym}+\left(G^\perp_t U_t\right)^{Sym}
\end{equation}
in the weak sense.
Consequently, by linearity of the Lebesgue decomposition,
\begin{equation}
\label{eq: expagtop}
\p_t G^\Lambda_t=\Lambda\left({g}_tU_t\right)^{Sym},
\end{equation}
\begin{equation}
\label{eq: expagbot}
\p_t G^\perp_t=\left(G^\perp_t U_t\right)^{Sym}.
\end{equation}
Moreover, $|\dot{G}_t|^2=\int_\Omega\rd\rhp_t U_t:U_t=\int_\Omega\rd \la g_t U_t:U_t+\int_\Omega\rd G^\perp_t U_t:U_t$.
Now we identify in \eqref{eq: pre1}
\begin{equation}
\label{eq: expa2}
\Lambda e^{-h(t)}(I-g_t)
=
\la\left[(I+e^{-h(t)}(g_t-I))\tilde V_t\right]^{Sym},
\end{equation}
where
$$
\tilde V_t:=\left(I+e^{-h(t)}(g_t-I)\right)^{-1}-I.
$$
Indeed,
$$
\left[\left(I+e^{-h(t)}(g_t-I)\right)\tilde V_t\right]^{Sym}=I-I-e^{-h(t)}(g_t-I)=e^{-h(t)}(I-g_t).
$$
Furthermore, for $\Lambda$-a.a. $x$ and a.a. $t$,
$$
(I+e^{-h(t)}(g_t(x)-I))\,\Xi_1:\Xi_2
$$
defines a real scalar product on the space of  matrices $\Xi\in \C^{d\times d}$.
Let $\Pi$ denote the corresponding orthogonal projection onto the subspace of Hermitian matrices (we omit the indexes and simply write $\Pi=\Pi_{t,x}$ for simplicity).
Then \eqref{eq: expagtop} implies
\begin{equation} 
\label{eq: expag1}
\p_t G_t^\Lambda=\Lambda\left[(I+e^{-h(t)}(g_t-I))\tilde W_t\right]^{Sym},
\end{equation}
where
$$
\tilde W_t:=\Pi\left(\left(I+e^{-h(t)}(g_t-I)\right)^{-1}g_tU_t\right).
$$

Now set
\begin{equation}\label{eq: pre3}
\tilde U_t:=\left(1-\mathbf{1}_{\mathrm{supp}\, G_t^\perp}\right)\left(h'(t)\tilde V_t+e^{-h(t)}\tilde W_t\right)+\mathbf{1}_{\mathrm{supp}\, G_t^\perp}\left(-h'(t)I+U_t\right).
\end{equation} 
From \eqref{eq: pre1} we see that 
\begin{equation}
\label{eq: pre2}
\p_t \tilde G_t^\Lambda-\Lambda\left[(I+e^{-h(t)}(g_t-I))(h'(t)\tilde V_t)\right]^{Sym}
=
\Lambda\left[(I+e^{-h(t)}(g_t-I))(e^{-h(t)}\tilde W_t)\right]^{Sym}.
\end{equation} 
From \eqref{eq: pre3} and \eqref{eq: pre2} we deduce that
\begin{equation} 
\label{eq: expa1top}
\p_t \tilde G^\Lambda_t=\Lambda\left[(I+e^{-h(t)}(g_t-I))\tilde U_t\right]^{Sym}
\end{equation}
For the singular part, \eqref{eq: pre1}, \eqref{eq:explicit_heat_flow_s} and \eqref{eq: expagbot} yield
\begin{equation}
\label{eq: pre7}
\p_t \tilde G_t^\perp+h'(t)\tilde G_t^\perp=(\tilde G_t^\perp U_t)^{Sym},
\end{equation}
and using  \eqref{eq: pre3} we conclude that
\begin{equation}
\label{eq: expa1bot}
\p_t \tilde G_t^\perp=(\tilde G_t^\perp  \tilde U_t)^{Sym}.
\end{equation}
Now \eqref{eq: pre3} gives
 \begin{multline}
 \label{eq: pre5}
 \frac 1 2 \int_\Omega \rd \tilde G_t\tilde U_t:\tilde U_t = \frac 1 2 \int_\Omega \rd \Lambda (I+e^{-h(t)}(g_t-I))\tilde U_t:\tilde U_t+\frac 1 2 \int_\Omega \rd \tilde \rhp_t^\perp\tilde U_t:\tilde U_t
 \\
 \leq  |h'(t)|^2\int_\Omega \rd \Lambda (I+e^{-h(t)}(g_t-I))\tilde V_t:\tilde V_t+  e^{-2h(t)}\int_\Omega \rd \Lambda  (I+e^{-h(t)}(g_t-I))\tilde W_t:\tilde W_t
 \\
 + |h'(t)|^2\int_\Omega \rd \tilde G_t^\perp I:I+\int_\Omega \rd \tilde G_t^\perp U_t:U_t
 \end{multline}
 for a.a. $t$.
An explicit computation shows moreover that
\begin{multline}
\label{e:calcf}
\int_\Omega \rd \Lambda (I+e^{-h(t)}(g_t-I))\tilde V_t:\tilde V_t+\int_\Omega \rd \tilde G_t^\perp I:I
\\
=\int_\Omega \rd \Lambda \left(I-(I+e^{-h(t)}(g_t-I))\right):\left(\left(I+e^{-h(t)}(g_t-I)\right)^{-1}-I\right)+\tr\int_\Omega \rd \tilde G_t^\perp
\\
=\int_\Omega \rd \Lambda \tr (\tilde g_t^{-1}+\tilde g_t-2I)+\tr\int_\Omega \rd \tilde G_t^\perp 
\\
= \int_\Omega \rd \Lambda \tr (\tilde g_t^{-1}-I) = F(\tilde G_t)
\end{multline}
since $\Lambda \tilde g_t+\tilde G_t^\perp=\tilde G_t$ has unit mass for a.a. $t$. Hence, the sum of the first and the third terms in the right-hand side of \eqref{eq: pre5} is $L^{\infty}_{loc}(0,1)$ by Lemma \ref{l:fbound}.

The second and the fourth term in the right-hand side of \eqref{eq: pre5} can
can be estimated in the following way:
\begin{multline}
\label{e:calcrhs}
e^{-2h(t)}\int_\Omega \rd \Lambda  (I+e^{-h(t)}(g_t-I))\tilde W_t:\tilde W_t+\int_\Omega \rd \tilde G_t^\perp U:U
\\
=  e^{-2h(t)}\int_\Omega \rd \Lambda  (I+e^{-h(t)}(g_t-I))\left[\Pi\left(\left(I+e^{-h(t)}(g_t-I)\right)^{-1}g_tU_t\right)\right]:\left[\Pi\left(\left(I+e^{-h(t)}(g_t-I)\right)^{-1}g_tU_t\right)\right]
\\
+e^{-h(t)}\int_\Omega \rd G_t^\perp U_t:U_t
\\
\leq  e^{-2h(t)}\int_\Omega \rd \Lambda  (I+e^{-h(t)}(g_t-I))\left[\left(I+e^{-h(t)}(g_t-I)\right)^{-1}g_tU_t\right]:\left[\left(I+e^{-h(t)}(g_t-I)\right)^{-1}g_tU_t\right]
\\
+e^{-h(t)}\int_\Omega \rd G_t^\perp U_t:U_t 
\\
= e^{-2h(t)}\int_\Omega \rd \Lambda  g_t(I+e^{-h(t)}(g_t-I))^{-1}g_tU_t:U_t + e^{-h(t)}\int_\Omega \rd G_t^\perp U_t:U_t
\\
\leq  e^{-h(t)}\int_\Omega \rd \Lambda  g_tU_t:U_t+e^{-h(t)}\int_\Omega \rd G_t^\perp U_t:U_t
=e^{-h(t)}\int_\Omega\rd\rhp_t U_t:U_t
=e^{-h(t)} |\dot G_t|^2.
\end{multline}
In the last inequality we have employed the purely algebraic inequality
\begin{equation} 
\label{e:algf}
A(I+\theta (A-I))^{-1}A\leq \theta^{-1}A
\end{equation}
for any positive-semidefinite matrix $A$ and $\theta \in (0,1)$.  
Since $G$ is an $AC^2$ curve we have $|\dot G|^2\in L^2(0,1)$ in \eqref{e:calcrhs}, the above calculations show that the right-hand side in \eqref{eq: pre5} is $L^2_{loc}$ in time, hence
$$
\tilde U_t\in L^2_{loc}(0,1;L^2(\rd {\tilde G}_t;\Sm)).
$$
Given that by construction $\p_t\tilde\rhp_t=(\tilde\rhp_t\tilde U_t)^{Sym}$ we conclude from Lemma \ref{l:ac2curves} that $\tilde G\in AC^2_{loc}((0,1); \Ma_{FR})$ and $\left|\dot{\tilde G}\right|^2=\int_\Omega \rd \tilde G \tilde U:\tilde U$ for a.a. $t\in(0,1)$. 

In view of Corollary \ref{cor:top} and in order to finally show the strong continuity of $\tilde G$ at the boundary (for definiteness, at $t=0$), we estimate
\begin{multline*}
\|\tilde G_t-\tilde G_0\|_{TV}
=\int_\Omega \rd \la |\tilde g_t-\tilde g_0|_2
+\|e^{-h(t)} G^\perp_t-e^{-h(0)}  G_0^\perp\|_{TV}
\\
=
\int_\Omega \rd \la \left|\left(I+e^{-h(t)}(g_t-I)\right)-\left(I+e^{-h(0)}(g_0-I)\right)\right|_2
\\
+ \left\|e^{-h(t)} (G^\perp_t-G^\perp_0)+(e^{-h(t)}-e^{-h(0)})  G_0^\perp\right\|_{TV}
\\
\leq |e^{-h(t)}-e^{-h(0)}| \int_\Omega\rd  \la \, |g_0-I|_2 + e^{-h(t)}\int_\Omega\rd  \la  |g_t-g_0|_2
\\
+e^{-h(t)}\| G^\perp_t- G_0^\perp\|_{TV}+|e^{-h(t)}-e^{-h(0)}|\times\|G^\perp_0\|_{TV}
\to 0 \end{multline*}
as $t\searrow 0$ and the proof is achieved.
\end{proof}

\begin{lem}[Chain rule]
\label{l:chain}
For $a.a. t\in (0,1)$ we have $\tr \tilde U_t\in L^1(\Omega;\rd \Lambda )$, where  $\tilde U_t$ is defined by \eqref{eq: pre3} and represents the dynamics of $\tilde\rhp$ through $\p_t \tilde G_t=(\tilde G_t  \tilde U_t)^{Sym}$.
Moreover $E(\tilde G_t)\in AC^2_{loc}((0,1);\R)$, and 
\begin{equation}
\label{eq:chain} 
\frac{d}{dt}\left(E(\tilde G_{t})\right) = -\int _\Omega\tr \tilde U_t \rd \Lambda
\end{equation} 
a.e.  in $(0,1)$.
\end{lem}
\begin{proof}
Recalling that  $\la I\in \Ma$ has unit mass and $\tr \tilde U$ is real, we observe that the Cauchy-Schwarz inequality gives for a.e. $t\in(0,1)$
\begin{multline}
\label{eq:boundtu}
\left(\int_\Omega \rd \la \tr \tilde U_t\right)^2 =\left(\int_\Omega \rd \la \tilde g_t \tilde U_t:\tilde g_t^{-1}\right)^2
\\
\leq 
\left(\int_\Omega \rd \la \tilde g_t \tilde U_t:\tilde U_t\right)\left(  \int_\Omega \rd \la \tilde g_t \tilde g_t^{-1}:\tilde g_t^{-1}\right)
=\left( \int_\Omega \rd \tilde G_t \tilde U_t:\tilde U_t\right) 
\left(\int_\Omega \rd \la \,\tr\left(\left[I+e^{-h(t)}(g_t-I) \right]^{-1}\right)\right)
\\
\leq 
\left(\int_\Omega \rd \tilde G_t \tilde U_t:\tilde U_t\right)
\left( \int_\Omega \rd \la \, \tr \left(\left[(1-e^{-h(t)})I) \right]^{-1}\right)\right)
=\frac 1 {1-e^{-h(t)}} \int_\Omega \rd \tilde G_t \tilde U_t:\tilde U_t
\end{multline}
Hence, $\tr \tilde U_t\in L^1(\Omega;\rd \Lambda )$ for a.a. $t$. 

It is well-known that the first variation of $g\mapsto \int_\Omega \rd\la \log \det g$ is $x\mapsto g^{-1}(x)$, as soon as $g$ is invertible $\Lambda$-a.e. in $\Omega$.
Since $\p_t\tilde G_t^\Lambda=\left({\tilde G}_t^\Lambda\tilde U_t\right)^{Sym}$ in the weak sense, see \eqref{eq: expa1top}, for every smooth test function $\psi:(0,1)\to \R$ we have, in the sense of distributions $\mathcal D'(0,1)$,
\begin{multline*}
\left\langle\frac d {dt} E(\tilde G), \psi\right\rangle=\left\langle\int_\Omega \rd\la \log \det (I+e^{-h(t)}(g-I)), \psi'\right\rangle 
\\ 
=-\left\langle \int_\Omega \left(\rd {\tilde G^\Lambda}\tilde U\right)^{Sym}:(I+e^{-h(t)}(g-I))^{-1}, \psi\right\rangle
\\
=-\left\langle \int_\Omega\rd \la  \left((I+e^{-h(t)}(g-I))\tilde U\right):(I+e^{-h(t)}(g-I))^{-1}, \psi\right\rangle 
\\
=-\left\langle \int_\Omega\rd \la \tr \tilde U, \psi\right\rangle. 
\end{multline*} 
This shows that \eqref{eq:chain} holds in the sense of distributions, but since the right-hand side is locally $L^2$ in time (due to \eqref{eq:boundtu} with $h(t)>0$ locally bounded away from zero), it also holds almost everywhere in $(0,1)$.
\end{proof}
The main ingredient in our construction of recovery sequences later on will be
\begin{lem}
For a.a. $t\in (0,1)$ we have
\begin{equation}
\label{eq:pre}
\frac{1}{2}\left|\dot{\tilde G}_t\right|^2
+
\frac{1}{2}|h'(t)|^2 F(\tilde G_{t})
+
h'(t)\frac{d}{dt}\left(E(\tilde G_{t})\right)
\leq \frac{1}{2}e^{-h(t)}\left|\dot G_t\right|^2\leq \frac{1}{2}\left|\dot G_t\right|^2.
\end{equation}
\end{lem}
Let us slightly anticipate at this stage that our construction of recovery sequences will consist in keeping the second Fisher information term and throwing away part of the third term on the left, while the proof of the geodesic convexity of the entropy will on the contrary be based on discarding $F\geq 0$ and integrating by parts $h'\frac {dE}{dt}$.
\begin{proof}
We use the same framework as in the proof of Lemma \ref{l:acreg} and keep the the same notations (for $\tilde U,\tilde V,\tilde W$).
From \eqref{eq: pre3} we compute 
\begin{multline}
\label{eq: pre4}
\frac 1 2 \int_\Omega \rd \Lambda \left[I+e^{-h(t)}(g_t-I)\right]\tilde U_t:\tilde U_t
\\
+ \frac 1 2 |h'(t)|^2\int_\Omega \rd \Lambda \left[I+e^{-h(t)}(g_t-I)\right]\tilde V_t:\tilde V_t
-  h'(t)\int_\Omega \rd \Lambda \left[I+e^{-h(t)}(g_t-I)\right]\tilde U_t:\tilde V_t
\\
= \frac 1 2 e^{-2h(t)}\int_\Omega \rd \Lambda  \left[I+e^{-h(t)}(g_t-I)\right]\tilde W_t:\tilde W_t,
\end{multline}
and 
\begin{equation}
\label{eq: pre8}
\frac 1 2 \int_\Omega \rd \tilde G^\perp_t \tilde U_t:\tilde U_t
+ \frac 1 2 |h'(t)|^2\int_\Omega \rd \tilde G_t^\perp I:I\\+ h'(t)\int_\Omega \rd \tilde G_t^\perp \tilde U_t:I
= \frac 1 2 \int_\Omega \rd \tilde G_t^\perp U_t:U_t.
\end{equation}
The sum of the left-hand sides of \eqref{eq: pre4} and \eqref{eq: pre8} is equal to the left-hand side of \eqref{eq:pre}.
Indeed,  the sum of the integrals in the first terms of \eqref{eq: pre4} and \eqref{eq: pre8} is exactly the squared metric derivative $\left|\dot{\tilde G}_t\right|^2=\int_\Omega\rd\tilde G_t\tilde U_t:\tilde U_t$ of $\tilde G$.
The sum of the second terms of \eqref{eq: pre4} and \eqref{eq: pre8} matches the second term of \eqref{eq:pre} due to \eqref{e:calcf}. 
Moreover,
\begin{multline*} 
-\int_\Omega \rd \Lambda \left[I+e^{-h(t)}(g_t-I)\right]\tilde U_t:\tilde V_t
+\int_\Omega \rd \tilde G^\perp_t \tilde U_t:I
\\
= -\int_\Omega \rd \Lambda \left[I+e^{-h(t)}(g_t-I)\right]\tilde U_t:\left(\left(I+e^{-h(t)}(g_t-I)\right)^{-1}-I\right)+ \tr \int_\Omega \rd \tilde G_t^\perp \tilde U_t
\\
=-\int \tr \tilde U_t \rd \Lambda+ \frac{d}{dt} \int \tr \tilde G_t^\Lambda+ \frac{d}{dt} \int \tr \tilde G_t^\perp= \frac{d}{dt}\left(E(\tilde G_t)\right).
\end{multline*}
We have used Lemma \ref{l:chain}, formulas \eqref{eq: expa1top} and \eqref{eq: expa1bot}, and the fact that the mass of $\tilde G_t\in \Ma$ is conserved.
We conclude by estimating the sum of the right-hand sides of \eqref{eq: pre4} and \eqref{eq: pre8} exactly as in \eqref{e:calcrhs}.
\end{proof}

The next result will essentially allow to integrate \eqref{eq:pre} in time all the way to $t=0,1$, which will be crucial later on but is {\it a priori} not legitimate so far since Lemma~\ref{l:acreg} and Lemma~\ref{l:chain} only give regularity \emph{locally} in $(0,1)$ at this stage.
\begin{lem}
\label{l:cont}
Assume that $h(0)=h(1)=0$ and that $h'(t)$ is bounded away from $0$ near $t=0$ and $t=1$.
Then $t\mapsto E(\tilde G_t)$ is continuous at the endpoints
 $0$ and $1$.
 \end{lem}
 \begin{proof}
 Since $t=0$ and $t=1$ are completely symmetric we only prove the statement at $0$.
 Observe from Lemma~\ref{l:acreg} that $\tilde\rhp_t\to\tilde\rhp_0=\rhp_0$ in $\dFR$ as $t\to 0^+$, or equivalently in total variation.
 Since $TV$ is stronger than the weak-$*$ convergence and due to the lower semi-continuity of the entropy $E$ for the weak-$*$ convergence \cite[thm. 2.34]{ambrosio2000functions}, we conclude immediately that $t\mapsto E(\tilde\rhp_t)$ is lower semicontinuous at $t=0^+$ and therefore it suffices to prove the upper semicontinuity.
 We only consider the case when $E(G_0)=E(\tilde G_0)$ is finite, otherwise there is nothing to prove.

 {\it Step 1:} Assume first that $\det g_0(x)\geq C>0$ is bounded away from $0$.
The curve $G$ is $AC^2$ with values in $\Ma$, thus by \eqref{eq:comparison_2sides_dFR_TV} and Lemma~\ref{lequiv} it is also $AC^2$ with values in the Banach space $\Smv_{TV}$.
By Morrey's embedding \cite{AK18} we see that $G:[0,1]\to \Smv_{TV}$ is $\frac 1 2$-H\"older continuous.
Thus,
$$
\int_{\Omega} \rd \la |g_0(x)-g_t(x)|_2=
\|\rhp_0^\Lambda-\rhp_t^\Lambda\|_{TV}
\leq
\|\rhp_0-\rhp_t\|_{TV}\leq 
C t^{1/2}.
$$
Consequently
$
\int_{[|g_0-g_t|_2>t^{1/4}]}\rd \la \leq Ct^{1/4}
$,
whence 
\begin{equation} 
\label{e:compres}
\int_{[g_0-g_t>\frac {t^{1/4}} {\sqrt d} I]}\rd \la \leq Ct^{1/4}.
\end{equation}
Then 
\begin{multline}
\label{eq:estimate_limsup_E}
E(\tilde G_t)=-\int_\Omega \rd\la \log \det \left[I+e^{-h(t)}(g_t-I)\right]
\\
=-\int_{[g_0-g_t>\frac {t^{1/4}} {\sqrt d} I]} \rd\la \log \det \left[I+e^{-h(t)}(g_t-I)\right]
-\int_{[g_0-g_t\leq \frac {t^{1/4}} {\sqrt d} I]} \rd\la \log \det \left[I+e^{-h(t)}(g_t-I)\right]
\\
\leq -\int_{[g_0-g_t>\frac {t^{1/4}} {\sqrt d} I]} \rd\la \log \det \left[(1-e^{-h(t)})I\right]-\int_{[g_0-g_t\leq \frac {t^{1/4}} {\sqrt d} I]} \rd\la \log \det (e^{-h(t)}g_t)
\\
\leq  -d\log (1-e^{-h(t)})\int_{[g_0-g_t>\frac {t^{1/4}} {\sqrt d} I]} \rd\la
- \int_{[g_0-g_t\leq \frac {t^{1/4}} {\sqrt d} I]} \rd\la \log \det \left[e^{-h(t)}\left(g_0-\frac {t^{1/4}} {\sqrt d} I\right)\right]
\\
\leq  -Ct^{1/4} \log (1-e^{-h(t)})
+\Bigg\{
\int_{[g_0-g_t> \frac {t^{1/4}} {\sqrt d} I]} \rd\la \log \det \left[e^{-h(t)}\left(g_0-\frac {t^{1/4}} {\sqrt d} I\right)\right]
\\
-\int_{\Omega} \rd\la \log \det \left[e^{-h(t)}\left(g_0-\frac {t^{1/4}} {\sqrt d} I\right)\right]\Bigg\}
\\
\leq 
-Ct^{1/4} \log (1-e^{-h(t)})
-\int_{\Omega} \rd\la \log \det \left[e^{-h(t)}\left(g_0-\frac {t^{1/4}} {\sqrt d} I\right)\right]\\
+
\int_{[g_0-g_t>\frac {t^{1/4}} {\sqrt d} I]} \rd\la \log \det g_0.
\end{multline}
The last inequality simply follows from $e^{-h(t)}(g_0-\frac {t^{1/4}} {\sqrt d} I)\leq e^{-h(t)}g_0\leq g_0$, and we exploited on several occasions that $\det g_0$ is bounded away from zero to guarantee that $g_0-\frac {t^{1/4}} {\sqrt d} I$ remains positive definite at least for small times.
The first term in the r.h.s. of \eqref{eq:estimate_limsup_E} behaves as $t^\frac 14|\log h(t)|$, which tends to zero as $t\to 0^+$ due to our current assumptions on $h$.
The second term tends to $-\int_{\Omega} \rd\la \log \det g_0=E(\tilde G_0)$.
(Our temporary assumption that $\det g_0$ is bounded away from zero allows to apply Lebesgue's dominated convergence theorem and we omit the details.)
Finally, the third member tends to zero due to \eqref{e:compres} and $\log\det g_0\in L^1(\rd\Lambda)$, by absolute continuity of the integral.

{\it Step 2:} for general $g_0$ we argue now by approximation.
Consider the curve
$$
(G^n)_t:=\S_{\frac 1 n}(G_t),
$$ 
and let $(\tilde G^n)_t=\S_{h(t)}(G^n_t)$ be the perturbed curve constructed as before but starting now from $G^n$ instead of $G$.
By \eqref{eq:pre} with $h(t)\equiv \frac 1 n$,
\begin{equation}
\label{e:goodgn}
|\dot{G}^n_t|^2\leq e^{-\frac 1n } |\dot G_t|^2
\end{equation}
for a.a. $t\in (0,1)$ and therefore $G^n\in AC^2([0,1];\Ma)$. 
Obviously
$$
\det g_0^n(x)=\det \left[I+e^{-\frac 1 n} (g_0(x)-I)\right]\geq \det \left[(1-e^{-\frac 1 n}) I\right]
$$
is bounded away from $0$ (for fixed $n$), hence from step 1
\begin{equation}
\label{e:entrn}
E(\tilde G^n_0)\geq \limsup_{t\to 0^+} E(\tilde G^n_t).  
\end{equation}
Note that by Lemma~\ref{lem:entdecay} the integrand in the entropy $E(\tilde \rhp ^n_t)$ is nondecreasing in $n$ for any $x\in\Omega$ ($t$ is fixed here), thus by Beppo Levi's monotone convergence theorem
\begin{equation}
\label{e:entrn1}
\lim_{n\to \infty} E(\tilde G^n_t)=E(\tilde G_t)
\qqtext{for all}
t\in[0,1].
\end{equation}
Set 
$$
E_+^n(t):=\int_0^t \left[\frac d {ds} E(\tilde G^n_s)\right]_+\,ds,
\qquad
\ E_-^n(t):=E(\tilde G^n_t)-E_+^n(t).
$$
That $\left[\frac d {ds} E(\tilde G^n_s)\right]_+\geq 0$ is indeed integrable in this definition follows from \eqref{eq:pre} and our assumption that $h'(t)$ is bounded away from zero (and positive) near $t=0$.
Recall from Lemma~\ref{l:chain} that $t\mapsto E(\tilde G_t)$ is locally absolutely continuous, so the only scenario possibly contradicting the upper semi-continuity would be an initial upwards jump.
In the limit $n\to+\infty$ the ``good part'' $E_+^n$ accordingly contains the nondecreasing absolutely continuous part of $E(\tilde G^n_t)$, while the ``bad part'' $E_-^n$ contains the nonincreasing absolutely continuous part together with any possible (asymptotically) ``bad'' jump at $t=0^+$.
Dropping the first two non-negative terms in \eqref{eq:pre} and integrating in time from $0$ to small $t>0$ implies that 
\begin{equation*}
0\leq E_+^n(t)\leq \frac{1}{2\inf_{[0,t]} h'} \int_0^t |\dot{G}^n_s|^2\, \rd s
\leq C  \int_0^t|\dot G_s|^2\, \rd s,
\end{equation*} 
where the last inequality follows from \eqref{e:goodgn}.
Note that we have  employed the fact that $h'$ is positive and bounded away from $0$ near $t=0$.
Since $G$ is $AC^2$ in the right-hand side, we infer that
\begin{equation}
\label{eq:limsupt_limsupn_E+}
 \limsup_{t\to 0^+}\limsup_{n\to \infty} E_+^n(t)=0.
\end{equation}
It is clear that $E_-^n(t)$ is a non-increasing function at least for $t>0$.
 Using the elementary fact that
 $
 \lim_{n\to \infty} (a_n+b_n)=\liminf_{n\to \infty} a_n+\limsup_{n\to \infty} b_n
 $
 for any real sequences such that the first limit exists, we estimate
\begin{multline*}
\limsup_{t\to 0^+}  E(\tilde G_t)= \limsup_{t\to 0^+} \lim_{n\to \infty} E(\tilde G^n_t)
= \limsup_{t\to 0^+} \lim_{n\to \infty}\left( E^n_-(t) + E^n_+(t) \right)
\\
=
\limsup_{t\to 0^+} \left(\liminf_{n\to \infty} E_-^n(t)+\limsup_{n\to \infty} E_+^n(t)\right)
\\
\leq \limsup_{t\to 0^+} \liminf_{n\to \infty} E_-^n(t)+\limsup_{t\to 0^+}\limsup_{n\to \infty} E_+^n(t)
\\
\overset{\eqref{eq:limsupt_limsupn_E+}}{=} \limsup_{t\to 0^+} \liminf_{n\to \infty} E_-^n(t)
\leq
\sup_{t>0} \liminf_{n\to \infty} E_-^n(t).
\end{multline*}
Leveraging the time monotonicity $E^n_-\downarrow$ for $t>0$ we continue as
\begin{multline*}
 \sup_{t>0} \liminf_{n\to \infty} E_-^n(t)
 \leq \liminf_{n\to \infty} \sup_{t>0} E_-^n(t)
\\
=  \liminf_{n\to \infty} \limsup_{t\to 0^+} E_-^n(t)
\leq  \liminf_{n\to \infty} \limsup_{t\to 0^+}\left (E_-^n(t)+E_+^n(t)\right)
\\
=  \liminf_{n\to \infty} \limsup_{t\to 0^+} E(\tilde G^n_t)
\overset{\eqref{e:entrn}}{=}  \liminf_{n\to \infty} E(\tilde G^n_0)
\overset{\eqref{e:entrn1}}{=}E(\tilde G_0).
\end{multline*}
\end{proof}
We start now carefully choosing some specific $h(t)$ in order to retrieve quantitative information
\begin{lem}
\label{theo:upper_bound_recovery}
 For $\eps>0$ and a given path $G\in AC^2([0,1];\Ma_{FR})$ with $E(G_0),E(G_1)<+\infty$, let $G^\eps$ be the path obtained as
 $$
 \rhp^\eps_t:=\S_{h(t)} G_t 
 \qqtext{with}
 h(t):=\eps \min(t,1-t).
 $$
 Then $G^\eps\in AC^2([0,1];\Ma_{FR})$, $F(G^\eps)\in L^1(0,1)$, and
 \begin{equation}
  \label{eq:upper_bound_recovery}
\frac{1}{2}\int_0^1\left|\dot{G^\eps_t}\right|^2\, \rd t
+
\frac{\eps^2}{2}\int_0^1 F(G^\eps_t)\, \rd t
\leq \frac{1}{2}\int_0^1\left|\dot G_t\right|^2+\eps(E(G_0)+E(G_1)).
 \end{equation} 
\end{lem}
\begin{proof}
 Take \eqref{eq:pre} and integrate by parts separately on the time intervals $(\delta,1/2)$ and $(1/2,1-\delta)$ for some small $\delta>0$ to obtain
 \begin{equation}
  \label{eq:upper_bound_recovery1}
\frac{1}{2}\int_\delta^{1-\delta}\left|\dot{G^\eps_t}\right|^2\, \rd t+
\frac{\eps^2}{2}\int_\delta^{1-\delta} F(G^\eps_t)\, \rd t+ 2\eps E(G^\eps_{1/2})
\leq
\frac{1}{2}\int_\delta^{1-\delta}\left|\dot G_t\right|^2+\eps(E(G^\eps_\delta)+E(G^\eps_{1-\delta})).
 \end{equation}
 (Note that $h''(t)\equiv 0$ separately inside both intervals.)
 The odd-looking term $2\eps E(G^\eps_{1/2})$ arises from the two boundary terms at $t=1/2^\pm$ in the two integrations by parts, and can be safely discarded since it is non-negative. 
Let us point out that although this term is simply ignored here, it will be crucial later on in our proof of the geodesic convexity, Theorem~\ref{theo:1/2_convex}.
 
 Our assumption that $G\in AC^2$ together with Lemma \ref{l:cont} give a uniform bound for the right-hand side
 \begin{multline*}
 \limsup\limits_{\delta\to 0} \frac{1}{2}\int_\delta^{1-\delta}\left|\dot G_t\right|^2+\eps(E(G^\eps_\delta)+E(G^\eps_{1-\delta}))
 =
 \frac{1}{2}\int_0^{1}\left|\dot G_t\right|^2 + \eps(E(G^\eps_0)+E(G^\eps_{1}))
 \\
 =\frac{1}{2}\int_0^{1}\left|\dot G_t\right|^2 + \eps(E(G_0)+E(G_{1}))<+\infty
 \end{multline*}
 (since we assumed that the endpoints have finite entropy).
 As a consequence \eqref{eq:upper_bound_recovery1} holds with $\delta=0$, which is exactly our claim.
 \end{proof} 

We now have enough technical tools to prove the $\Gamma$-convergence.
For $G_0, G_1\in \Ma$ and $G\in C([0,1];\Ma)$, let
$$
\iota_{G_0, G_1} (G)=
\begin{cases}
0 & \mbox{if }G|_{t=0} =G_0 \mbox{ and }G|_{t=1} =G_1\\
+\infty & \mbox{otherwise}
\end{cases}
%
$$
be the convex indicator of the endpoint constraints.
For any $G\in C([0,1];\Ma)$ we define the kinetic action
$$
\mathcal K(G):=
\frac 1 2\int_0^1|\dot G_t|_{FR}^2\,\rd t ,
$$
 with the usual convention that $\mathcal K(G)=+\infty$ whenever $G\not\in AC^2([0,1];\Ma_{FR})$.
We also set 
$$
\mathcal F(G):=\int_0^1 F (G_t)\,\rd t
$$
where the integral may be infinite.

\begin{theo} \label{t:gc}
Let $G_0, G_1\in \Ma$ with $E(G_0),E(G_1)<+\infty$.
Then
\begin{equation}
\label{e:glim}
\Gamma-\lim_{\epsilon\to 0}  \left[ \mathcal K+ \frac {\epsilon^2} 2 \mathcal F + \iota_{G_0, G_1} \right]=\mathcal K  + \iota_{G_0, G_1}
\end{equation}
both in the strong uniform topology of $C([0,1];\Ma)$ and pointwise-in-time weak-$*$ topology of $\Ma\subset(C_0(\Omega;\Sm) )^*$.
\end{theo}

\begin{proof}
Since the strong uniform topology (w.r.t. to $TV$ or $\dFR$, equivalently) is stronger than the pointwise weak-$*$ topology, it suffices to prove the $\Gamma-\liminf$ for the latter topology and the $\Gamma-\limsup$ for the former one.

The $\Gamma-\liminf$ is quite obvious since $\mathcal K$ is lower-semicontinuous w.r.t. pointwise-in-time weak-$*$ topology (Proposition \ref{prop:AC2_energy_convex_LSC}), the convex indicator $\iota_{G_0, G_1}$ is also l.s.c., and $\mathcal K+\iota_{G_0, G_1}\leq \mathcal K+ \epsilon^2 \mathcal F + \iota_{G_0, G_1}$.

To prove the $\Gamma-\limsup$, for any sequence $\eps_n\to 0$ it suffices to construct a recovery sequence for \eqref{e:glim}. 
Moreover, we can restrict to curves $G$ such that the right-hand side of \eqref{e:glim} is finite, that is, $G\in AC^2([0,1];\Ma_{FR})$ with $G(0)=G_0,\ G(1)=G_1$ (otherwise there is noting to prove).
Thus, the desired recovery sequence $G^n\in C([0,1];\Ma)$ should satisfy
\begin{equation}
\label{recovery}
\limsup_{n\to \infty}\left\{\frac 1 2\int_0^1|\dot G^n_t|^2\rd t+\frac {\eps_n^2} 2\int_0^1 F (G^n_t) \rd t\right\}\leq \frac 1 2\int_0^1|\dot G_t|^2
\end{equation}  
with $G^n(0)=G_0$ and $G^n_1=G_1$.
By \eqref{eq:upper_bound_recovery}, the sequence $G^{\eps_n}$ from Lemma \ref{theo:upper_bound_recovery} does the job. Indeed, the time-continuity of $G^{\eps_n}$ follows from Lemma \ref{l:acreg}. 
It remains to check that $\rhp^{\eps_n}\to \rhp$ in the uniform topology of $C([0,1]; \Ma)$.
From the representation formula \eqref{eq:explicit_heat_flow} for solutions of the heat flow we compute explicitly 
\begin{multline*}
\|\rhp^{\eps_n}_t-\rhp_t\|_{TV}
=
\left\|\left(\Lambda I +e^{-\eps_n\min(t,1-t)}(\rhp_t-\Lambda I)\right) - \rhp_t\right\|_{TV}
\\
=\left(1-e^{-\eps_n\min(t,1-t)}\right)\|\rhp_t-\Lambda I\|_{TV}
\leq \left(1-e^{-\eps_n\min(t,1-t)}\right)\left(\|\rhp_t\|_{TV}  + \|\Lambda I\|_{TV}\right)
\\
= 2 \left(1-e^{-\eps_n\min(t,1-t)}\right)
\to 0
\end{multline*}
uniformly in $t$ as $\eps_n\to 0$. 
By the lower bound in \eqref{eq:comparison_2sides_dFR_TV} and Lemma \ref{lequiv} we conclude that
$$
\dihs^2(\rhp_t^{\eps_n},\rhp_t)\leq C
\dih^2(\rhp_t^{\eps_n},\rhp_t)
\leq 
C \|\rhp_t^{\eps_n}-\rhp_t\|_{TV}\to 0
$$
uniformly in $t$ as $\eps_n\to 0$, and the proof is complete.
\end{proof}
%
%
As an immediate and natural consequence we have that the $\eps$-geodesics (minimizers of $\mathcal K+\eps^2\mathcal F$) converge to FR-geodesics (minimizers of $\mathcal K$):
\begin{cor}
\label{c:gc}
Let $\eps_k\searrow 0$ and $G^k$ be the corresponding solution of the Schr\"odinger problem \eqref{e:miniy} with $\eps=\eps_k$ and $E(G_0), \, E(G_1)<\infty$.
Then there exists $G\in C([0,1];\Ma)$ such that, up to a subsequence,
$$
G^k_t\xrightarrow[k\to\infty]{}G_t
\qquad\mbox{weakly-}*
$$
for every $t\in [0,1]$, and $G$ solves the geodesic problem \eqref{e:minifr}.
Moreover,
$$
2\mathfrak S_{\eps_k}(G_0,G_1)\to \dFR^2(G_0,G_1).
$$
\end{cor}
\begin{proof}
Recall that $\Gamma$-convergence precisely guarantees that limits of minimizers are minimizers, \cite[Theorem 1.21]{Braides}, thus in view of Theorem~\ref{t:gc} it suffices to prove that the set of minimizers $\{G^k\}$  is relatively compact in the pointwise-in-time weak-$*$ topology.
Indeed, the $AC^2$-energies of the curves $G^k$ are uniformly bounded since
$$
\mathcal K(G^k)\leq \mathcal K(G^k)+ \frac {\epsilon_k^2} 2 \mathcal F(G^k) \leq \mathcal K(G^1)+ \frac {\epsilon_k^2} 2 \mathcal F(G^1)  \leq \mathcal K(G^1)+ \frac {\epsilon_1^2} 2 \mathcal F(G^1)<+\infty.
$$
By the fundamental estimate \eqref{eq:control_TV_leq_sqrt_FR} with $m^k_t=M=1$ on $\Ma$ we get
$$
\forall\,t,s\in [0,1],\,\forall k\in \N:\qquad \|\rhp^k_s-\rhp^k_t\|_{TV}\leq C|t-s|^{1/2}.
$$
Arguing as in the proof of Lemma \ref{lem:dihs_LSC_weak*}, we deduce that there exists a $TV$-continuous curve $(\rhp_t)_{t\in [0,1]}$ connecting $\rhp_0$ and $\rhp_1$ such that
\begin{equation*}
\forall t\in [0,1]:\qquad \rhp^k_t\to \rhp_t\quad\mbox{ weakly-}*
\end{equation*}
along some subsequence $k\to\infty$. 
\end{proof}
\begin{remark}If $E(\rhp_0)$ or $E(\rhp_1)$ are infinite, a careful regularization (involving of course the heat flow) allows to prove that the $\eps_n$-geodesics with suitably regularized endpoints $\rhp_0^n,\rhp_1^n\to \rhp_0,\rhp_1$ still converge to the geodesic with endpoints $\rhp_0,\rhp_1$.
The statement and proof both become slightly more involved (essentially one should make sure to regularize enough so that $\eps_n[E(\rhp^n_0)+E(\rhp^n_1)]\to 0$), and we omit the details for the sake of brevity.\end{remark}

We finish this section with a slightly different and perhaps unexpected consequence of our previous construction of the regularized curves:
\begin{theo}
\label{theo:1/2_convex}
 The entropy $E:\Ma\to \R^+$ is $\frac 12$-geodesically convex for the Fisher-Rao distance, namely for any $\rhp_0,\rhp_1\in\Ma$ and any FR geodesic $(\rhp_\theta)_{\theta\in[0,1]}$ joining $\rhp_0,\rhp_1$ there holds
 \begin{equation}
  E(\rhp_{\theta})\leq (1-\theta) E(\rhp_0) + \theta E(\rhp_1) -\frac 14\theta(1-\theta)\dFR^2(\rhp_0,\rhp_1)
  \qqtext{for all}
  \theta\in(0,1).
 \end{equation}
\end{theo}

\begin{proof}
Observe that our statement is vacuous if either of the endpoints $\rhp_0,\rhp_1$ has infinite entropy, thus we need only consider $E(\rhp_0),E(\rhp_1)<+\infty$.
Pick any geodesic $(\rhp_t)_{t\in[0,1]}$, fix $\theta\in(0,1)$, and let
$$
 H_\theta(t):=
 \left\{
 \begin{array}{ll}
  \displaystyle{\frac{1}{\theta}} t & \mbox{if }t\in [0,\theta],\\
  - \displaystyle{\frac{1}{1-\theta}(t-1)} & \mbox{if }t\in [\theta,1]
 \end{array}
 \right.
 $$
 be the hat function centered at $t=\theta$ with height $1$ and vanishing at the boundaries.
 Setting $h(t):=\eps H_\theta(t)$ for small $\eps>0$, let
 $$
 \rhp^\eps_t:= \S_{h(t)}\rhp_t
 $$
 be the perturbed curve constucted as before.
 Integrating \eqref{eq:pre} in time separately on $[0,\theta]$ and $[\theta,1]$, discarding the nonnegative Fisher information term, and leveraging the continuity of the entropy from Lemma~\ref{l:cont} (with $G^\eps_t=G_t$ at the endpoints $t=0,1$), we get
 $$
 \frac{1}{2}\int_0^1\left|\dot \rhp^\eps_t\right|^2\rd t
 +0
 +
\frac\eps\theta  \left[E(\rhp^\eps_\theta)-E(\rhp_0) \right]
+
\frac{\eps}{1-\theta}  \left[E(\rhp^\eps_\theta)-E(\rhp_1) \right]
\leq
\frac{1}{2}\int_0^1 e^{-\eps H_\theta(t)}\left|\dot \rhp_t\right|^2\rd t.
 $$
 Because $\rhp^\eps$ is an admissible curve connecting $\rhp_0,\rhp_1$ the first term in the left-hand side is larger than the $AC^2$ energy of the minimizing geodesic $\rhp$: Multiplying by $\theta(1-\theta)/\eps>0$ and rearranging gives
 $$
  E(\rhp^\eps_{\theta})\leq (1-\theta) E(\rhp_0) + \theta E(\rhp_1) +\frac {\theta(1-\theta)}{2}\int_0^1\frac{e^{-\eps H_\theta(t)}-1}{\eps}|\dot \rhp_t|^2\rd t.
 $$
 Since $\rhp^\eps \to \rhp$ weakly-$*$ for any fixed time, the lower semicontinuity of the entropy allows to take the $\liminf$ as $\eps\to 0$ in the left-hand side.
 For the right-hand side, observe that because $\rhp$ is a minimizing geodesic it has constant speed, $|\dot \rhp_t|^2=cst =\dFR^2(\rhp_0,\rhp_1)$.
 As a consequence we get
 \begin{multline*}
  E(\rhp_{\theta})
  \leq \liminf\limits_{\eps\to 0} E(\rhp^\eps_{\theta})
  \\
  \leq (1-\theta) E(\rhp_0) + \theta E(\rhp_1) +\lim\limits_{\eps\to 0} \frac {\theta(1-\theta)}{2}\dFR^2(\rhp_0,\rhp_1)\int_0^1\frac{e^{-\eps H_\theta(t)}-1}{\eps}\rd t
  \\
  = (1-\theta) E(\rhp_0) + \theta E(\rhp_1) - \frac {\theta(1-\theta)}{2}\dFR^2(\rhp_0,\rhp_1)\int_0^1H_\theta(t)\rd t.
 \end{multline*}
Our statement finally follows from $\int_0^1 H_\theta(t)\rd t=\frac 12$ for all $\theta$.
\end{proof}
\appendix
\section{Some technical lemmas}
\label{sec:appendix}
\begin{lem}[Constant-speed reparametrization]
\label{lem:constant_speed_reparametrization}
 Let $(\rhp_t,U_t)$ be a curve connecting $\rhp_0,\rhp_1$ with finite energy
 $$
 E[\rhp,U]=\int_0^1\int_\Omega\rd\rhp_t U_t:U_t\rd t<+\infty.
 $$
 Then there exists a curve $(\check\rhp_t,\check U_t)_{t\in[0,1]}$ connecting $\check\rhp_0=\rhp_0$ to $\check\rhp_1=\rhp_1$ with
 $$
 \|\check U_t\|_{L^2(\rd\check\rhp_t)}\equiv cst
 \qqtext{and}
 E[\check\rhp,\check U]
 = \left(\int_0^1\|U_t\|_{L^2(\rd\rhp_t)}\rd t\right)^2
 \leq E[\rhp,U]
 $$
with strict inequality unless $\|U_t\|_{L^2(\rd\rhp_t)}$ is constant in time.
\end{lem}
\begin{proof}
The argument is fairly standard (see e.g. \cite[lemma 5.3]{KMV16A} or \cite[lemma 1.1.4]{AGS06}) hence we only sketch the idea.
Consider the change of time variable
$$
\mathsf s(t):=\int_0^t \|U_\tau\|_{L^2(\rd\rhp_\tau)}\rd\tau,
\qqtext{and}
\mathsf t(s):= \min\{t\in[0,1]:\quad \mathsf s(t)=s\}
$$
($\mathsf t$ is the left inverse of $\mathsf s$).
Setting
$$
\widetilde \rhp_s:=\rhp_{\mathsf t(s)},
\qquad 
\widetilde U_s:=\frac{U_{\mathsf t(s)}}{\|U_{\mathsf t(s)}\|_{L^2(\widetilde \rhp_s)}}
$$
gives an admissible path connecting $\widetilde\rhp_0=\rhp_0$ to $\widetilde\rhp_L=\rhp_1$ in time $s\in[0,L]$ with
$$
L=\mathsf s(1)=\int_0^1 \|U_\tau\|_{L^2(\rd\rhp_\tau)}\rd\tau
$$
and unit speed $\|\widetilde U_s\|_{L^2(\rd\widetilde\rhp_s)}\equiv 1$.
This is clear at least formally from the chain-rule
$$
\p_s\widetilde \rhp_s
=
\frac{\rd \mathsf t}{\rd s}(s) \p_t\rhp_t(\mathsf t(s))
=
\frac{1}{\frac{\rd \mathsf s}{\rd t}(\mathsf t(s))}\left( \rhp_{\mathsf t(s)} U_{\mathsf t(s)}\right)^{Sym}
=
\frac{1}{ \|U_{\mathsf t(s)}\|_{L^2(\rd\rhp_{\mathsf t(s)})}}\left( \rhp_{\mathsf t(s)} U_{\mathsf t(s)}\right)^{Sym}
=(\widetilde \rhp_s \widetilde U_s)^{Sym}
$$
and can be made rigorous since the denominator $\|U_{\mathsf t(s)}\|_{L^2(\rd\rhp_{\mathsf t(s)})}$ only vanishes at the discontinuity points of $\mathsf t(s)$, which are countable (being $\mathsf t$ monotone nondecreasing) and therefore $\rd s$-negligible.
Scaling $t=Ls$
$$
(\check \rhp_t,\check U_t):=(\widetilde\rhp_{tL},L\widetilde U_{tL})
\qquad \mbox{for }t\in[0,1]
$$
back to the unit interval and noticing that $\|\check U_t\|_{L^2(\rd\check\rhp_t)}\equiv L$ is constant, we get
\begin{multline*}
E[\check \rhp,\check U]
=
\int_0^1 \|\check U_t\|^2_{L^2(\rd\check\rhp_t)} \,\rd t
=
\int_0^1 L^2\rd t=L^2
\\
=\left(\int_0^1\|U_\tau\|_{L^2(\rd\rhp_\tau)}\rd \tau\right)^2
\leq 
\int_0^1\|U_\tau\|_{L^2(\rd\rhp_\tau)}^2\rd \tau
=
E[\rhp,U]
\end{multline*}
as desired.
(The inequality is strict unless $(\rhp,U)$ has constant speed as in our statement.)
\end{proof}
\begin{lem}[Refined Banach-Alaoglu \cite{KMV16A}]
\label{Ban}
Let $(X,\|\cdot\|)$ be a separable normed vector space. Assume that there exists a sequence of seminorms $\{\|\cdot\|_k\}$ ($k=0,1,2,\dots$) on $X$ such that for every $x\in X$ one has 
$$
\|x\|_k\leq C \|x\|
$$
with a constant $C$ independent of $k,x$, and
$$
\|x\|_k\underset{k\to\infty}{\rightarrow} \|x\|_0.
$$
Let $\varphi_k$ ($k=1,2,\dots$) be a uniformly bounded sequence of linear continuous functionals on $(X,\|\cdot\|_k)$, resp., in the sense that 
$$
c_k:=\|\varphi_k\|_{(X,\|\cdot\|_k)^*}\leq C.
$$
Then the sequence $\{\varphi_k\}$ admits a converging subsequence $\varphi_{k_n}\to \varphi_0$ in the weak-$*$ topology of $X^*$, and
\begin{equation} \label{e:liminfp}
\|\varphi_0\|_{(X,\|\cdot\|_0)^*}\leq c_0:=\liminf\limits_k c_{k}.
\end{equation}
\end{lem} 
%
%
%
%
\begin{lem}[Refined Arzel\`a-Ascoli \cite{AGS06,BV18}]
\label{L:aa}
Let $(X,\varrho)$ be a metric space. Assume that there exists a Hausdorff topology $\sigma$ on $X$ such that $\varrho$ is sequentially lower semicontinuous with respect to $\sigma$.  Let $(x^k)_t$, $t\in[0,1]$, be a sequence of curves lying in a common $\sigma$-sequentially compact set $K\subset X$.  Let it be equicontinuous in the sense that there exists a symmetric continuous function $\omega:[0,1]\times [0,1]\to\R_+$, $\omega(t,t)=0$, such that \begin{equation}
\label{eccurve}
\varrho((x^k)_t,(x^k)_{\bar t})\leq \omega(t,\bar t).
\end{equation} for all $t,\bar t\in[0,1]$. Then there exists a $\varrho$-continuous curve $x_t$ such that \begin{equation}
\label{eccurvelim}
\varrho(x_t,x_{\bar t})\leq \omega(t,\bar t),
\end{equation} and (up to a not relabelled subsequence) \begin{equation}
\label{conv}
(x^k)_t\to x_t 
\end{equation} for all $t\in[0,1]$ in the topology $\sigma$.
\end{lem}

%
%
%
%
\begin{proof}[Proof of Proposition \ref{d:bur}] 
First of all, let us show that the right-hand side is always finite and that the minimum is always attained.
For any fixed $A_1\in\Po$ it is easy to check that $(A_t,U_t):=(t^2A_1,\frac 2t I)$ gives an admissible path connecting $A_0=0$ to $A_1$ with finite energy.
In particular any two matrices $A_0,A_1\in\Po$ can be connected through zero as $A_0\leadsto 0\leadsto A_1$ with finite cost, thus the problem is proper.
For fixed $A_0,A_1$ consider now a minimizing sequence $(A^n_t,U^n_t)_{t\in[0,1]}$.
Note that our Lemma~\ref{lem:mass_estimate_TV} applies in particular when $\Omega=\{x\}$ is a one-point space, which gives here equicontinuity and pointwise relative compactness in the form
$$
|A^n_t-A^n_s|_2\leq C |t-s|^\frac 12,
\qquad m^n_t=\tr A^n_t\leq M,
\qquad s,t\in[0,1]
$$
uniformly in $n$.
By the classical Arzel\'a-Ascoli theorem we get, up to extraction of a subsequence if needed,
$$
A^n\to A
\qquad\mbox{uniformly in }C([0,1];\Po).
$$
This immediately shows that the matrix-valued measure $\mu^n:=A^n_t\rd t\to A_t dt=:\mu$ at least weakly-$*$ on $\Mm(0,1)$.
Because
$
\|U^n\|_{L^2(\rd \mu^n)}^2=\int_0^1 A^n_tU^n_t:U^n_t\,\rd t\leq C
$,
an easy application of our Banach-Alaoglu variant (lemma~\ref{Ban}) in varying $L^2(\rd\mu^n)$ spaces gives a limit $U\in L^2(\rd \mu)$ with
$$
\int_0^1 A_t U_t:U_t\rd t
\leq \liminf\limits_{n\to\infty}
\int_0^1 A^n_tU^n_t:U^n_t\,\rd t
$$
and such that $\int_0^1 A^n U^n:V\,\rd t\to \int_0^1 A U:V\,\rd t$ for any reasonably smooth test function $V$.
This shows that this limit $(A_t,U_t)_{t\in[0,1]}$ is an admissible curve joining $A_0,A_1$ with energy
$
E[A,U]\leq \liminf E[A^n,U^n]
$
and this pair is therefore a minimizer.

In order to identify now the left-hand side and the right-hand side of \eqref{e:minibur} we proceed in two steps.\\
{\it Step 1:} assume that $A_0,A_1\in\Pm(d)$ are positive definite, and let $\check A_{0,1}:=\mathfrak r(A_{0,1})$ be the corresponding real extensions as defined in \eqref{eq:def_inclusion_r}.
From Proposition~\ref{d:burwc} there holds $d_B^2(A_0,A_1)=\frac 12 W^2_2(\mathcal N(\check A_0),\mathcal N(\check A_1))$.
In this real setting it is known \cite[Prop. A]{takatsu2011wasserstein} that
$$
W^2_2(\mathcal N(\check A_0),\mathcal N(\check A_1)) = \frac 1 4 \min\limits_{\check A\in\mathcal A(\check A_0,\check A_1)} \int_0^1 \check A_t \check U_t:\check U_t\,\rd t,
$$
where the infimum runs of course over real pairs $\frac{d\check A_t}{dt}=(\check A_t \check U_t)^{Sym}$ with $\check U_t\in \mathcal S(2d)$.
Complexifying back $(\check A,\check U)\leadsto (A,U)$ gives the result.
\\
{\it Step 2: }if now either $A_0$ or $A_1$ are only \emph{semi}-definite we approximate $A^n_0:=A_0+\frac 1n I\to A_0$ and $A^n_1:=A_1+\frac 1n I\to A_1$.
Clearly $A^n_{0,1}\in \Pm$ are positive-definite, hence step 1 applies.

Note from \eqref{e:burexp} that the left-hand side of \eqref{e:minibur} is of course continuous in $A_0,A_1$, hence it suffices to show that the optimal value in the right-hand side is continuous for this particular choice of $A^n_{0,1}\to A_{0,1}$.
Observe that
$$
\mathcal J^*(A_0,A_1):=\frac 1 4 \min_{\mathcal{A}(A_0,A_1)}\int_0^1 A_t U_t :U_t \rd t
$$
is a well-defined function of $A_0,A_1\in \Po$, and we proved earlier that there always exists a minimizer.
Arguing precisely as for the existence of the minimizers, cf. also Lemma \ref{lem:dihs_LSC_weak*}, it is easy to prove that $\mathcal J^*$ is lower semi-continuous in both arguments.
Indeed, up to a subsequence, any sequence $(A^n_t,U^n_t)_{t\in[0,1]}$ of minimizers in $\mathcal J^*(A_0^n,A_1^n)$ converges to an admissible candidate $(A_t,U_t)_{t\in[0,1]}$ connecting $A_0,A_1$, hence $\mathcal J^*(A_0,A_1)\leq E[A,U]\leq \liminf E[A^n,U^n]=\liminf \mathcal J^*(A^n_0,A^n_1)$
(regardless of the particular form of $A_0^n,A_1^n$).

In order to establish the upper continuity, let $A_0=R D_0 R^*$ be a spectral decomposition of $A_0$ and note that obviously $A^n_0=A_0+\frac 1n I=R(D_0+\frac 1n I)R^*$.
Since $A^n_0$ and $A_0$ commute it is easy to check that
%
$$
A^n_{0t}:=R\left[(1-t)\sqrt{D_0}+t\sqrt{D_0+\frac 1n I}\right]^2R^*,
\qquad
U^n_{0t}:=2 R\left(\sqrt{D_0+\frac 1n I}-\sqrt{D_0}\right)R^*\sqrt{A_{0t}^{-1}}
$$
defines an admissible path $(A^n_{0t})_{t\in[0,1]}$ between $A_0$ and $A^n_0$.
Moreover,
a straightforward computation shows that the corresponding energy is
\begin{equation}
\label{eq:appendix_energy_to_0}
\frac 1 4 \int_0^1A^n_{0t} U^n_{0t}:U^n_{0t}\,\rd t
=
\left|\sqrt{D_0}-\sqrt{D_0+I/n}\right|^2_2
\xrightarrow[n\to\infty]{} 0.
\end{equation}
(Actually this path is exactly the Bures geodesic between $A_0,A_0^n$.)
A similar construction yields a path $(A^n_{1t})_{t\in[0,1]}$ connecting $A_1^n$ to $A_1$ with cost
\begin{equation}
\label{eq:appendix_energy_to_01}
\frac 1 4 \int_0^1A^n_{1t} U^n_{1t}:U^n_{1t}\,\rd t
=
\left|\sqrt{D_1}-\sqrt{D_1+I/n}\right|^2_2
\xrightarrow[n\to\infty]{} 0,
\end{equation}
where $D_1$ is the spectral decomposition of $A_1=QD_1Q^*$.
Now pick a minimizer $(\tilde A_t,\tilde U_t)_{t\in[0,1]}$ in the definition of $\mathcal J^*(A_0,A_1)$
and fix a small $\theta_n\in(0,1)$ to be determined shortly.
Rescaling in time and concatenating the paths $A_0^n\leadsto A_0\leadsto A_1\leadsto A_1^n$ in the intervals $t\in[0,\theta_n]$, $t\in[\theta_n,1-\theta_n]$, and $t\in[1-\theta_n,1]$, respectively, we obtain an admissible path $(\hat A^n,\hat U^n)$ from $A_0^n$ to $A_1^n$ whose energy is bounded as
\begin{multline*}
4 \mathcal J^*(A_0^n,A_1^n)\leq \int _0^1\hat A_t^n\hat U^n_t:\hat U^n_t\,\rd t
\\
= \frac{1}{\theta_n} \int_0^1A^n_{0t} U^n_{0t}:U^n_{0t}\,\rd t +\frac{1}{1-2\theta_n} \int_0^1 \tilde A_t\tilde U_t:\tilde U_t\,\rd t
+\frac{1}{\theta_n} \int_0^1A^n_{1t} U^n_{1t}:U^n_{1t}\,\rd t.
\end{multline*}
By \eqref{eq:appendix_energy_to_0}, \eqref{eq:appendix_energy_to_01} we see that the first and third integrals in the right-hand side tend to zero, while the second integral is exactly $4 \mathcal J^*(A_0,A_1)$ by definition of $(\tilde A,\tilde U)$.
Choosing $\theta_n\to 0$ sufficiently slowly and taking $\limsup$ in the previous inequality gives
$$
\limsup_{n \to \infty}  4\mathcal J^*(A_0^n,A_1^n)\leq 0+\lim_{n\to \infty}\frac{1}{1-2\theta_n}\int_0^1 \tilde A_t\tilde U_t:\tilde U_t\,\rd t + 0
=
4\mathcal J^*(A_0,A_1)
$$
and the proof is complete.
\end{proof}
%
%
%
%
\begin{proof}[Proof of Proposition~\ref{prop:Schrodinger_Gaussian}]
Let $\rho_0=\mathcal N(A_0)$ and $\rho_1=\mathcal N(A_1)$.
By \cite[Thm. 3.3 and Thm. 3.4]{leonard2013survey}, if one could write the $(f,g)$ transform
\begin{equation}
\label{eq:Schrodinger_system}
\begin{cases}
\rho_0=f_0g_0\\
\rho_1=f_1g_1
\end{cases}
\end{equation}
for $(f_t)_{t\in[0,1]}$ a \emph{forward} solution of the heat equation with initial datum $f_0$ and $(g_t)_{t\in[0,1]}$ a \emph{backward} solution with terminal datum $g_1$, then the solution of \eqref{e:oldy} would be given by
$$
\rho_t=f_tg_t.
$$
Since the product of Gaussian distributions is Gaussian, it is legitimate to try and solve for $f_0=\mathcal N(B_0)$ and $g_1=\mathcal N(C_1)$ as Gaussians.
Since the heat flow for Gaussians is explicitly given by  \eqref{e:gfb1} we see that the corresponding forward and backward solutions read
$$
\begin{cases}
f_t=\mathcal N(B_t) & \qquad \mbox{with } B_t=B_0+2t I
\\
g_t=\mathcal N(C_t) & \qquad \mbox{with } C_t=C_1+2(1-t) I
\end{cases}
$$
Exploiting the algebraic product rule $\mathcal N(B)\times\mathcal N(C)=\mathcal N\left([B^{-1}+C^{-1}]^{-1}\right)$, we see that, given $A_0,A_1$ the Schr\"odinger system \eqref{eq:Schrodinger_system} is equivalent to solving
$$
\begin{cases}
\mathcal N(A_0)=\mathcal N(B_0)\mathcal N(C_0)\\
\mathcal N(A_1)=\mathcal N(B_1)\mathcal N(C_1)
\end{cases}
\qquad
\Leftrightarrow
\qquad
\begin{cases}
 A_0^{-1}= B_0^{-1}+[C_1+2I]^{-1}
 \\
 A_1^{-1} = [B_0+2I]^{-1}+ C_1^{-1}
\end{cases}
$$
in $B_0,C_1\in \Ppmre$.
It is then a simple exercise to check that this system has a unique solution, and in particular the covariance $A_t$ of $\rho_t=f_tg_t$ is fully determined by
$$
A_t^{-1}=[B_0+2tI]^{-1} +[C_1+2(1-t)I]^{-1}.
$$
\end{proof}
\section*{Acknowledgments}
Credit is due to Aymeric Baradat for our construction of recovery sequences, which was improved and adapted from \cite{baradat2020small}. We are very grateful to the anonymous referee for pointing out the link with the theory of $C^*$-algebras, cf. Remark \ref{noncom2}. 
LM wishes to thank Jean-Claude Zambrini for numerous and fruitful discussions on the Schr\"odinger problem, and acknowledges support from the Portuguese Science Foundation through FCT project PTDC/MAT-STA/28812/2017 {\it Schr\"oMoka}.  DV was partially supported by the FCT projects UID/MAT/00324/2020 and PTDC/MAT-PUR/28686/2017.

\end{document}